\documentclass[11pt,reqno]{amsart}

\usepackage[T1]{fontenc}
\usepackage[english]{babel}

\usepackage{amsmath,amssymb,amscd,amsfonts,mathtools,tikz,caption,float,tikz-cd}
\usetikzlibrary{patterns}
\usepackage[linkcolor=blue,colorlinks=true,citecolor=blue]{hyperref}
\usepackage[capitalise]{cleveref}
\usepackage{colonequals}
\usepackage{wrapfig}

\newtheorem{thm}{Theorem}[section]
\newtheorem{lm}[thm]{Lemma}
\newtheorem{defn}[thm]{Definition}
\newtheorem{prop}[thm]{Proposition}
\newtheorem{coro}[thm]{Corollary}
\newtheorem{rmk}[thm]{Remark}
\newtheorem{conj}[thm]{Conjecture}

\newcommand{\scal}[1]{\left\langle#1\right\rangle}

\newcommand{\op}[1]{\left(#1\right)}

\newcommand{\eps}{\epsilon}

\newcommand{\bbm}{\begin{bmatrix}}
\newcommand{\ebm}{\end{bmatrix}}
\newcommand{\bpm}{\begin{pmatrix}}
\newcommand{\epm}{\end{pmatrix}}
\newcommand{\bsm}{\left(\begin{smallmatrix}}
\newcommand{\esm}{\end{smallmatrix}\right)}
\newcommand{\bsbm}{\left[\begin{smallmatrix}}
\newcommand{\esbm}{\end{smallmatrix}\right]}

\newcommand\ba{\begin{align*}}
\newcommand\ea{\end{align*}}
\newcommand\baa{\begin{align}}
\newcommand\eaa{\end{align}}
\newcommand\ben{\begin{enumerate}}
\newcommand\een{\end{enumerate}}

\newcommand{\figref}[1]{\hyperref[#1]{Figure \ref{#1}}}
\newcommand{\lemref}[1]{\hyperref[#1]{Lemma \ref{#1}}}
\newcommand{\thmref}[1]{\hyperref[#1]{Theorem \ref{#1}}}
\newcommand{\conjref}[1]{\hyperref[#1]{Conjecture \ref{#1}}}
\newcommand{\propref}[1]{\hyperref[#1]{Proposition \ref{#1}}}
\newcommand{\corref}[1]{\hyperref[#1]{Corollary \ref{#1}}}
\newcommand{\defref}[1]{\hyperref[#1]{Definition \ref{#1}}}
\newcommand{\rmkref}[1]{\hyperref[#1]{Remark \ref{#1}}}
\newcommand{\qref}[1]{\hyperref[#1]{Question \ref{#1}}}
\newcommand{\secref}[1]{\hyperref[#1]{\S\ref{#1}}}
\newcommand{\appref}[1]{\hyperref[#1]{Appendix \ref{#1}}}

\newcommand{\R}{\mathbf{R}}
\newcommand{\C}{\mathbf{C}}

\newcommand{\h}{\mathbf{H}}
\newcommand{\Q}{\mathbf{Q}}
\newcommand{\Z}{\mathbf{Z}}
\newcommand{\N}{\mathbf{N}}

\newcommand{\GL}{\mathrm{GL}}

\newcommand{\cB}{\mathcal{B}}

\newcommand{\cE}{\mathcal{E}}

\newcommand{\cH}{\mathcal{H}}

\newcommand{\cN}{\mathcal{N}}

\newcommand{\cS}{\mathcal{S}}

\newcommand{\ga}{\alpha}     
\newcommand{\gl}{\lambda}    

\renewcommand{\hat}{\widehat}

\definecolor{blue}{rgb}{0,0,1}
\definecolor{red}{rgb}{1,0,0}
\definecolor{green}{rgb}{0,.6,.2}
\definecolor{purple}{rgb}{1,0,1}

\numberwithin{equation}{section}

\title[Rational points on the sphere]{Equidistribution, covering radius, and Diophantine approximation for rational points on the sphere}
\author{Claire Burrin}
\address{Institute of Mathematics, University of Zurich}
\email{claire.burrin@math.uzh.ch}

\author{Matthias Gr\"obner}
\address{Institute of Mathematics, EPFL}
\email{matthias.grobner@epfl.ch}

\thanks{We are grateful to Tim Browning for answering some questions. CB is supported by Swiss National Science Foundation Grants No.~201557 and No.~10003145. MG is supported by Swiss National Science Foundation Grant No.~215337.}

\begin{document}

\begin{abstract}
We study the distribution of rational points of fixed height on the sphere at shrinking scales. For the two-dimensional sphere, we prove an unconditional variance estimate for primitive square-level Linnik sets, essentially matching the random-model prediction. We obtain almost-everywhere equidistribution in caps down to the optimal scale $R\gg n^{-1/2+\delta}$, pointwise equidistribution down to $R\gg n^{-1/4+o(1)}$, and Wasserstein equidistribution down to the optimal bound. We also derive applications to covering, intrinsic Diophantine approximation, and Linnik's conjecture on sums of two squares and a mini-square.

\end{abstract}

\maketitle

\section{Introduction}

Understanding the distribution of point sets on the sphere is a classical problem with wide-ranging applications, including numerical integration and approximation, quantum information theory, and molecular structures. Depending on the purpose, different measures are used to assess what constitutes a good distribution, the most common being equidistribution, covering and packing radii, and minimal energy. 

\begin{figure}[ht]
  \centering
  \includegraphics[scale=.25]{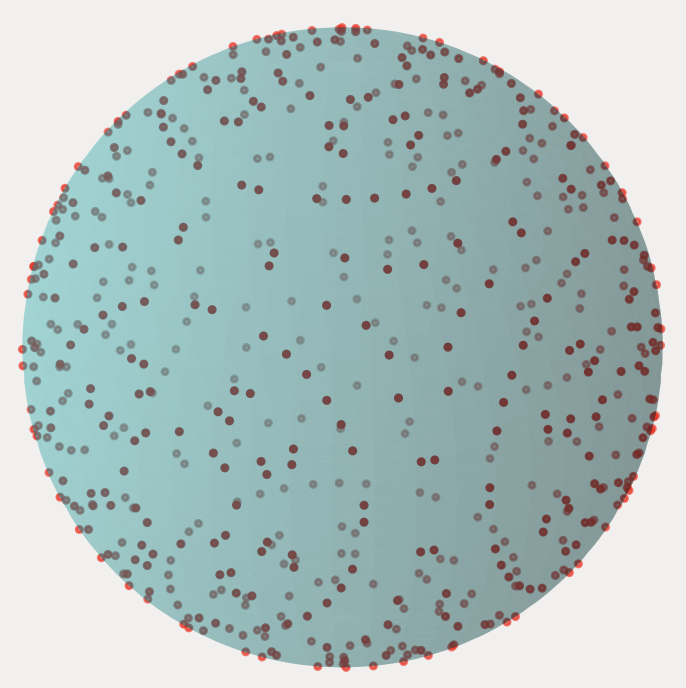} \includegraphics[scale=.25]{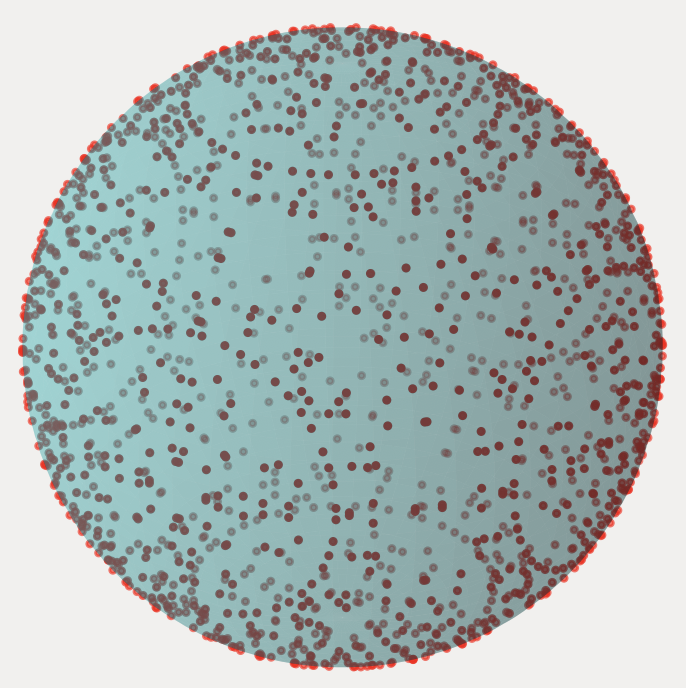} \includegraphics[scale=.25]{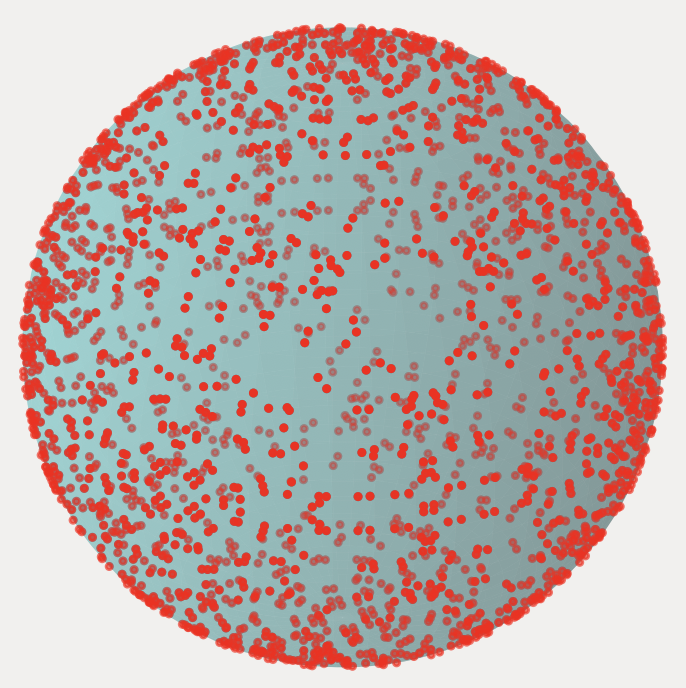} \includegraphics[scale=.227]{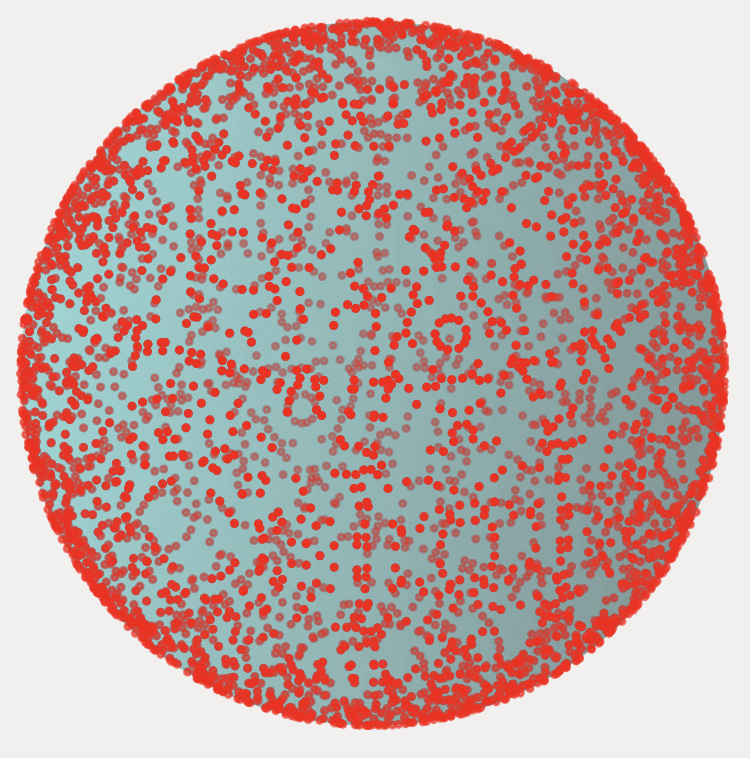}
  \caption{$\Omega_n\subset S^2$ for $n=101,201,401,701$.}\label{fig:Omega_n}
  \end{figure}

In a series of influential works, Bourgain, Rudnick, and Sarnak promoted the Linnik sets
$$
\cE_n = \{ n^{-1/2}(x,y,z) \mid (x,y,z)\in\Z^3,\, x^2+y^2+z^2 =n \} \subset S^2, \quad n\neq 4^a(8b+7)
$$
as a rich arithmetic model of pseudorandom points on the sphere. They studied their local statistics, variance in shrinking regions, and covering properties \cite{bourgain2012local,bourgain2016spatial},  proving that,  possibly under certain arithmetic hypotheses such as GRH, GLH, or the absence of Siegel zeros, these sets behave like random points sampled independently according to the normalized surface area measure $\mu$. One such result is that the variance for Linnik points in shrinking caps is as small as for a random set.

\begin{thm}\label{thm:BRS}\cite{bourgain2016spatial}
    Let $C_R(\alpha)$ be a spherical cap of radius $R$, centered at $\alpha\in S^2$. Assume  the Lindelöf Hypothesis for standard $\GL(2)/\Q$ $L$-functions. Then for squarefree $n\not\equiv 7$ (mod 8), we have 
    \begin{align*}
        \int_{S^2} \left( |\cE_n\cap C_R(\alpha)| -|\cE_n|\mu(C_R)\right)^2 d\mu(\alpha) \ll n^\eps |\cE_n|\mu(C_R)
    \end{align*}
for any $\eps>0$.
\end{thm}

The point-sets we will consider in this article are rational points on $S^2$ of fixed height, given by
\begin{align*}
    \Omega_n = \{ n^{-1}(x,y,z) \mid (x,y,z)\in\Z^3,\, x^2+y^2+z^2=n^2,\, {\rm gcd}(x,y,z,n)=1\}.
\end{align*}
We restrict to odd $n$, since $\Omega_n$ is empty if and only if $n$ is even. These sets form a natural primitive square-level subsequence of the Linnik sets. Our first result is an unconditional analogue of the variance estimate for Linnik points.

\begin{thm}\label{thm:variance}
    Let $n$ be odd. Then 
        \begin{align*}
        \int_{S^2} \left( |\Omega_n\cap C_R(\alpha)| -|\Omega_n|\mu(C_R)\right)^2 d\mu(\alpha) \ll n^{\eps} |\Omega_n|\mu(C_R)
    \end{align*}
for any $\eps>0$.
\end{thm}

Since $|\Omega_n|=n^{1+o(1)}$ (when $n$ is odd), the average number of points in a cap of radius $R$ is $\asymp nR^2$. Thus $R\asymp n^{-1/2}$ is the critical
scale at which a cap contains only constantly many points on average.
The variance bound implies almost-everywhere equidistribution down to the optimal scale.

\begin{coro}\label{coro:a.e. optimal}
For every $\delta>0$, outside a set of centers $\alpha$ of measure $o(n^{-\delta})$, we have
\begin{align*}
    \frac{|\Omega_n\cap C_R(\alpha)|}{|\Omega_n|\mu(C_R)}\to 1
\end{align*}
uniformly for $R\gg n^{-1/2+\delta}$.
\end{coro}

The method of proof also yields the following pointwise small-scale equidistribution.

\begin{thm}\label{thm:dim 2}
For every center $\alpha\in S^2$, odd $n$,
\begin{align*}
    \frac{|\Omega_n\cap C_R(\alpha)|}{|\Omega_n|\mu(C_R)}\to 1
\end{align*}
whenever $R\gg n^{-1/4+\eps}$, for any $\eps>0$.
\end{thm}

The same spectral second moment also gives an optimal quantitative
equidistribution statement against Lipschitz test functions. Let
$$
\mu_n:=\frac{1}{|\Omega_n|}\sum_{x\in\Omega_n}\delta_x
$$
be the discrete probability measure supported on $\Omega_n$, and let $W_1$ denote
the 1-Wasserstein distance on $S^2$, with respect to the geodesic
distance. By Kantorovich--Rubinstein duality \cite{Villani2003}, this is 
$$
W_1(\mu_n,\mu) = \sup_{{\rm Lip}(f)\leq 1} \left|
\frac{1}{|\Omega_n|}\sum_{x\in\Omega_n}f(x)-\int_{S^2}f\,d\mu
\right|.
$$
In transport terms, the quantity $W_1(\mu_n,\mu)$ measures how far, on average, surface mass must move to reach $\Omega_n$. The following statement shows that this distance is optimal, up to a factor $n^{o(1)}$. 

\begin{thm}\label{thm:wasserstein}
For every odd \(n\), we have
\[
W_1(\mu_n,\mu)\ll n^{-1/2+o(1)}.
\]
This rate is optimal, up to the $n^{o(1)}$ factor, among all probability
measures supported on $|\Omega_n|$ points.
\end{thm}

This result is weaker than cap discrepancy, since cap indicators are discontinuous, but it is a natural quantitative metric for pseudorandom point configurations. Recent work of Sanchez Garza \cite{SanchezGarza2026} uses the 1-Wasserstein distance in a closely related arithmetic setting, proving quantitative equidistribution for arithmetic spherical harmonics conditionally on GLH.

The restriction to square-levels is essential. A spectral expansion argument reduces the variance and equidistribution problems to bounding Fourier coefficients of half-integral weight theta series. For general Linnik levels, Waldspurger's formula relates these coefficients to central values of automorphic $L$-functions, and this is the source of the Lindel\"of-type hypotheses in the work of Bourgain, Rudnick, and Sarnak, and the subsequent work of Humphries and Radziwi\l\l\,  \cite{HumphriesRadziwill2022}. At square levels, Shimura's correspondence relates the size of these coefficients to that of coefficients of holomorphic cusp forms of integer weight, for which the Ramanujan--Petersson conjecture holds. In dimension two, a further arithmetic input comes from the quaternionic Hecke symmetries of $S^2$ arising from the Hamilton quaternion algebra. These form a commutative algebra of Hecke operators on $L^2(S^2)$ and their simultaneous eigenforms are related to classical cuspidal newforms via an explicit realization of the Jacquet--Langlands correspondence. Choosing a simultaneous Hecke basis avoids the dimension loss in the usual theta-series argument and yields the variance estimate above.

This argument also yields a conditional extension of the variance theorem of Bourgain--Rudnick--Sarnak to all admissible levels. Indeed, every $n\not\equiv 0,4,7$ (mod 8) admits a decomposition $n=d\ell^2$, where $d$ is squarefree and $d\not\equiv7$ (mod 8), and $\ell$ is odd. The squarefree part $d$ is the fundamental level entering Waldspurger's formula, while the odd square factor $\ell^2$ is controlled by the argument sketched above.

\begin{thm}\label{thm:BRS all levels}
Under the Lindel\"of Hypothesis for standard $\GL(2)/\Q$ $L$-functions, the variance estimate of \cref{thm:BRS} holds for every $n\not\equiv 0,4,7$ (mod 8).
\end{thm}

For completeness, and to put the $S^2$ result in context, we also record the following uniform small-scale equidistribution theorem in all dimensions.

\begin{thm}\label{thm:small scale}
For every $d\geq  2$ and every spherical cap $C_R(\alpha)\subset S^d$, there exists a set $\cN_d\subset\N$ of positive density such that
\begin{align*}
    \frac{|\Omega_n\cap C_R(\alpha)|}{|\Omega_n|\mu(C_R)} \to 1
\end{align*}
as $n \to \infty, n \in \mathcal N_d$, whenever $R\gg n^{-(d-1)/(2d+1)+ \eps}$.
\end{thm}

Small-scale equidistribution for rational points on the sphere is naturally connected to the covering radius problem and to intrinsic Diophantine approximation. We review these connections and the current state-of-the-art in the rest of this introduction. We also discuss progress towards a stronger version of the following conjecture of Linnik, asking for the existence of primitive solutions \cite[Chapter XI]{linnik1968ergodic}. 

\begin{conj}[Linnik]
  For each $n\not\equiv 0,4,7$ (mod 8), the Diophantine equation $x^2 +y^2+z^2=n$ is solvable in the integers with $|z|=O(n^\eps)$ (for any $\epsilon >0$).
 \end{conj}

\subsection{Covering exponents and Diophantine approximation}

Given a point configuration $X_N\subset S^d$, the covering radius is the least radius $R=R(X_N)$ so that the sphere is covered by spherical caps $C_R$ of radius $R$ centered at the points in $X_N$, i.e.,
$$
S^d = \bigcup_{x\in X_N}  \overline{C_R(x)}.
$$ 
Equivalently, the covering radius is the largest distance from a point in $S^d$ to the nearest point in $X_N$. Given a sequence of point sets $(X_N)_N \subset S^d$, their asymptotic covering rate can be captured by the covering exponent \cite{sarnak2015letter}
\begin{align}\label{eq:def K}
 K((X_N)) \coloneqq  - \liminf_{N\to\infty} \frac{\log |X_N|}{\log \mu(C_{R(X_N)})},
\end{align}
where $|X_N|$ denotes the cardinality of $X_N$. A sequence of point sets $(X_N)_N$ is considered optimally distributed (with respect to covering) if the covering exponent is 1, or equivalently, if the covering radii satisfy $R(X_N,S^d) = |X_N|^{-1/d + o(1)}$ as $N \rightarrow \infty$. Random point-sets on the sphere and spherical $t$-designs with $N_t \asymp t^d$ points \cite{BondarenkoRadchenkoViazovska2013}, for example, are optimally distributed with respect to covering \cite{ReznikovSaff2015,Yudin1995}.

A small covering exponent  for the sequence $\Omega_T$ of all rational points of height (smallest common denominator in reduced form) up to $T$   implies that for any (irrational) point on the sphere, there exists a rational point on the sphere sufficiently close to it. The sizes of these finite sets are of order $|\Omega_T|\asymp T^d$ for $T\gg1$.

We then speak of intrinsic Diophantine approximation. This theme, put forward by Lang in his influential ``Report on Diophantine Approximation'' \cite{Lang1965}, picked up traction in the last two decades as it became apparent that often the best rational approximants are in fact intrinsic.  As in the classical theory, badly approximable points on the sphere are of measure zero; this motivates the following definition.\footnote{ Our definition differs from that of average covering exponent introduced by Sarnak \cite{sarnak2015letter}.}

\begin{defn}[Generic Covering Exponent]
Let $(X_N)_N \subset S^d$ be a sequence of point sets. We define the generic covering exponent as
  \begin{align*}
    K_\mu((X_N)) \coloneqq  \inf_{(R_N)_N} \left( - \liminf_{N\to\infty} \frac{\log |X_N|}{\log \mu(C_{R_N})} \right),
  \end{align*}
  where the infimum is taken over all sequences of radii $(R_N)_N$ such that $\mu \big( S^d\setminus \cup_{x\in X_N} C_{R_N}(x) \big) \to 0$.  
\end{defn}

Thus $K_\mu$ measures the smallest possible asymptotic exponent among radius sequences that cover almost all of the sphere. 

By now, analogues of classical Diophantine approximation results (Dirichlet, Khintchine, Schmidt) have been established for spheres \cite{KleinbockMerrill2015,AlamGhosh2022,KelmerYu2023,ouaggag2024central}. These results imply the following strong results for the covering exponent.

\begin{prop}[Covering Exponents for $(\Omega_T)$]\label{thm:Diophantine}
  For every $d\geq2$ we have
  $
  K((\Omega_T))=2$ and $ K_\mu((\Omega_T))=1.$
\end{prop}

\begin{figure}[ht]
  \centering
  \includegraphics[scale=.25]{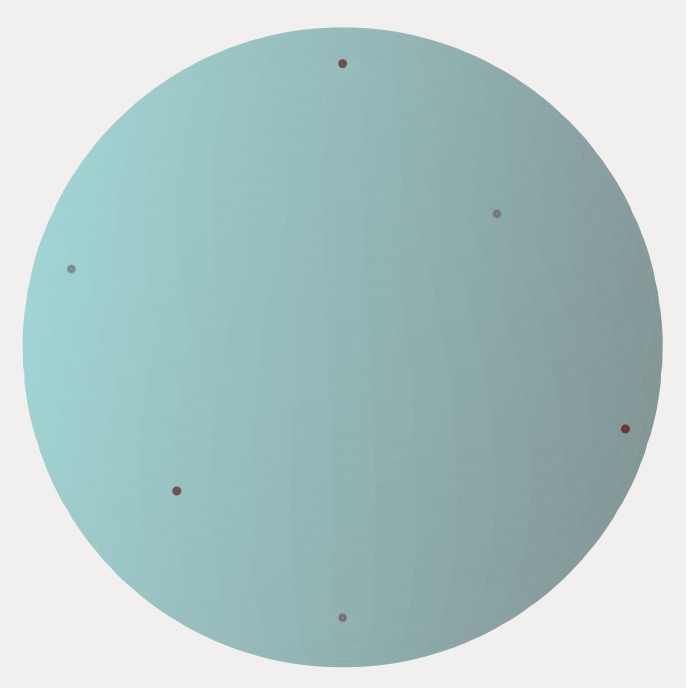} \includegraphics[scale=.25]{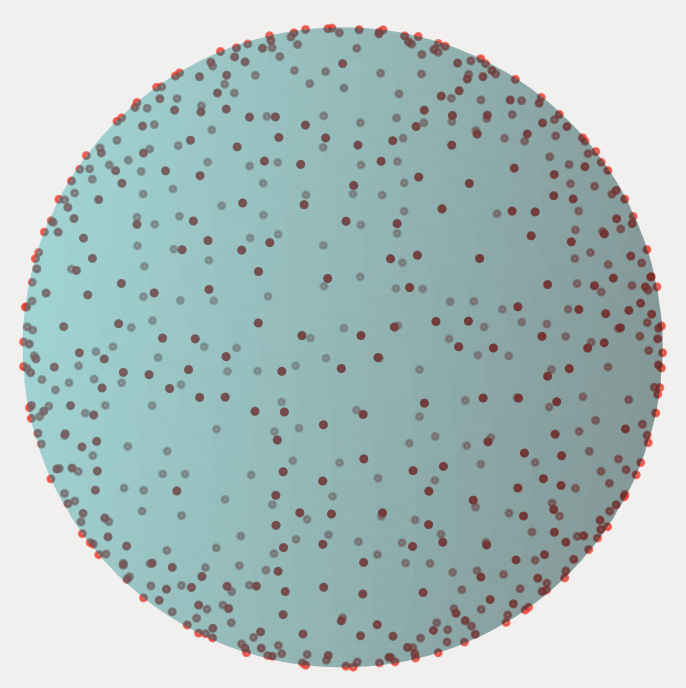} \includegraphics[scale=.25]{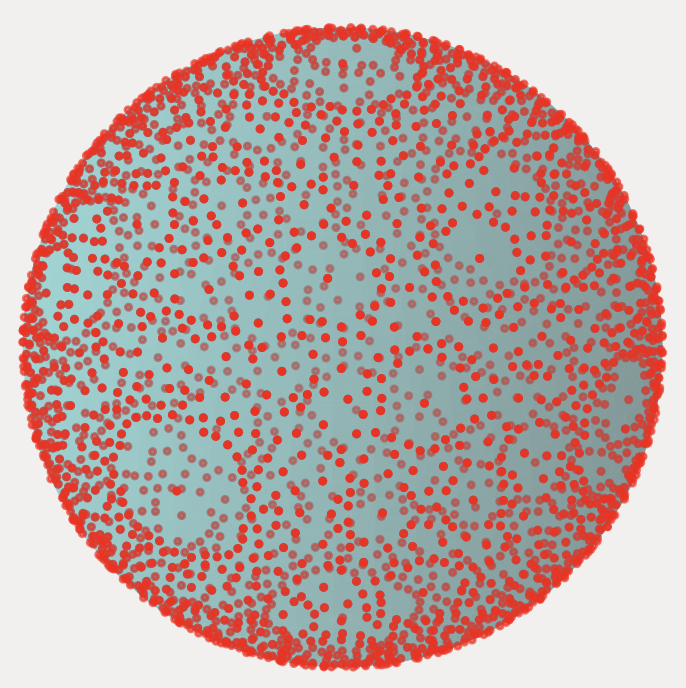} \includegraphics[scale=.25]{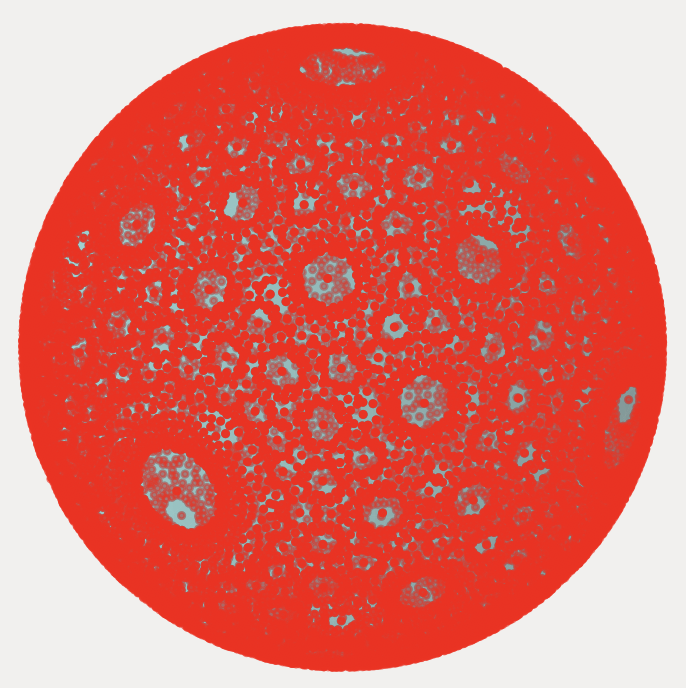}
  \caption{$\Omega_T\subset S^2$ for $T=1,20,40,100$.}\label{fig:Omega_T}
  \end{figure}
  
The proof of \cref{thm:Diophantine} shows that the lower bound $K((\Omega _T))\geq 2$ reflects the fact that rational points of small heights are badly approximable by other rational points; this leads to the `holes' we see emerging in \cref{fig:Omega_T}. In contrast, the second part of the statement shows that these holes, though responsible for the worst-case covering radius, occupy an asymptotically negligible proportion of the sphere. 

A similar phenomenon is observed when restricting to the sequence $(\Omega_n)_{n}$ of rational points of individual heights. Namely, it follows from \cref{thm:small scale} that
\begin{align*}
  2\left(1-\frac{1}{d}\right) \leq K((\Omega_n)) \leq 2+\frac{1}{d},
\end{align*}
where the lower bound follows from the simple observation that any $x\in \Omega_n$ (with $n>1$) is at distance  $\gg n^{-1/2}$  from the north pole. Hence in large dimensions, the covering exponent tends to 2, and in fact striking recent applications of the Hardy--Littlewood circle method imply that $K=2(1-\tfrac{1}{d})$ for $d\geq 4$ \cite{Sardari2019} and $K=\tfrac{4}{3}$ when $d=3$, conditionally on a twisted variant of Linnik's conjecture for Kloosterman sums \cite{browning2019twisted}.  

The case of the 2-dimensional sphere stands out, as only in this case is an optimal covering of the sphere by $(\Omega _n)$ in principle achievable. That $1\leq K\leq 2$ while $K_\mu=1$ --- as implied by \cref{coro:a.e. optimal} --- was anticipated by Sarnak \cite{sarnak2015letter}, based on the same ideas we exploit here. It would be remarkable if the covering exponent is optimal in dimension 2.  The arithmetic repulsion near low-height  rational points produces empty caps of radius $\asymp n^{-1/2}$, which is precisely the scale allowed by optimal covering on $S^2$.   In this sense, these holes are compatible with the optimal $K((\Omega_n))=1$.  We come back to this point in connection with Linnik's conjecture on Diophantine equations at the end of this introduction.


We emphasize that the measure of small-scale equidistribution is more rigid than that of the covering radius, which only asks for a lower bound on $R$ for $|\Omega_n \cap C_R|>0$, whereas small-scale equidistribution demands that $|\Omega_n\cap C_R|\to\infty$ as $n\to\infty$. Small-scale equidistribution allows in particular to quantify the abundance of rational approximants in intrinsic Diophantine approximation \cite{KelmerYu2023}. The analogue of Dirichlet's theorem states that for some constant $c>0$, $\psi(n):=c n^{-1}$, and any irrational $\alpha$ there exist infinitely many $n$ such that $|\Omega_n\cap C_{\psi(n)}(\alpha )|>0$ \cite{KleinbockMerrill2015}, and a quantitative refinement (analogue of a theorem of Schmidt) asserts that for almost every $\alpha$ we have on average $|\Omega_n \cap C_{c/n}(\alpha )| \asymp \tfrac{1}{n}$ \cite{AlamGhosh2022}. In contrast, \cref{coro:a.e. optimal} underlines the abundance of rational approximants in caps with a radius greater than $R\gg n^{-1/2+\eps}$, for every fixed $\eps>0$.

\begin{coro}\label{coro:quantitative Diophantine approximation}
Fix $\tau<\tfrac12$, $c>0$, and set $\psi(n)=c n^{-\tau}$. For every odd $n$, outside an exceptional set $E_n\subset S^2$ with $\mu(E_n)=o(n^{-\delta})$ (for any $\delta<1/2-\tau$), we have uniformly for $\alpha\in S^2\setminus E_n$,
$$
|\Omega _n \cap C_{\psi(n)}(\alpha )| = \frac{3c^2}{2} n^{1-2\tau} \prod_{p\mid n}(1-\chi(p)p^{-1})(1+o(1)),
$$
where $\chi$ is the primitive character (mod 4).
\end{coro}
\begin{proof}
    The proof combines \cref{coro:a.e. optimal} with $\mu(C_{\psi(n)}) = \tfrac{c^2 n^{-2\tau}}{4} +O(n^{-4\tau})$ (\cref{lm:area cap}) and $|\Omega_n|=6n\prod_{p\mid n}(1-\chi(p)p^{-1})$ for $n$ odd (\cref{sec:classical stuff}).
\end{proof}

\subsection{Linnik's conjecture on sums of squares and a mini-square}

\cref{thm:dim 2} also yields some progress towards the following conjecture of Linnik: for $n\not\equiv 0,4,7$ (mod 8), the Diophantine equation $x^2 +y^2+z^2=n$ is solvable in the integers with $|z|=O(n^\eps)$.

Note that the conjecture is easily proved for sums of four or more squares by the classical sum-of-squares theorems of Lagrange and Legendre. The condition $n\not\equiv0,4,7$ (mod 8) is implied by Legendre's theorem: a number $n$ is representable as a sum of three squares if and only if $n\neq 4^a(8b+7)$ for $a,b\geq0$. Linnik's conjecture is known to be true for almost all admissible $n$ by work of Wooley \cite{wooley2014linnik}, while Golubeva and Fomenko showed that it is true for all admissible $n$ with $|z|=O(n^{1/2-7/705+\epsilon })$ \cite{golubeva1994conjecture}. Stronger results have been established for specific subsets of $n\not\equiv0,4,7$ (mod 8): Humphries and Radziwi\l \l\, recently showed that for all squarefree $n\equiv 3$ (mod 8) there is a solution with $|z|=O(n^{4/9+\epsilon })$ or $|z|=O(n^{1/3+\epsilon })$ if we admit the General Lindel\"of Hypothesis \cite{HumphriesRadziwill2022}, while Golubeva--Fomenko obtained $|z|=O(n^{2/5+\epsilon })$ for all $n=d\ell^2\equiv1,3,5$ (mod 8) ($d\geq1$). For perfect squares $n=\ell^2$  the conjecture is trivially satisfied by $(x,y,z)=(\ell,0,0)$. However, if we restrict Linnik's question to primitive solutions, the equation is not solvable whenever $\ell$ is even.  A modification of the proof of \cref{thm:dim 2} (replacing caps by annuli) leads to the following result.
\begin{thm}\label{thm:Linnik}
For every odd perfect square $n$, the Diophantine equation $x^2+y^2+z^2=n$ has a {\em primitive} solution in the integers with $|z|=O(n^{1/3+\eps})$.
\end{thm}

`Trivial' solutions are provided by primitive Pythagorean triples; whenever $\ell$ is the hypotenuse of a right triangle with integer side-lengths --- equivalently, whenever $\ell$ is an odd number with no prime factor $\equiv 3$ (mod 4) --- we obtain a primitive solution with $z=0$. However the set of numbers with no prime factor $\equiv 3$ (mod 4) has density zero. In view of \cref{thm:Linnik}, a bound of $|z|=O(n^{1/4+\eps})$ would follow if small-scale equidistribution for caps of radius $R\asymp \ell^{-1/2+\eps}$ (which \cref{coro:a.e. optimal} establishes to hold for almost every cap on the sphere) were true for spherical caps about the north pole (or any other pole). Since the largest 'holes' in the distribution of rational points on the sphere occur precisely near the coordinate poles, these locations remain the natural obstruction in approaching Linnik's conjectural bound. 

\begin{figure}[ht]
  \centering
  \includegraphics[scale=.35]{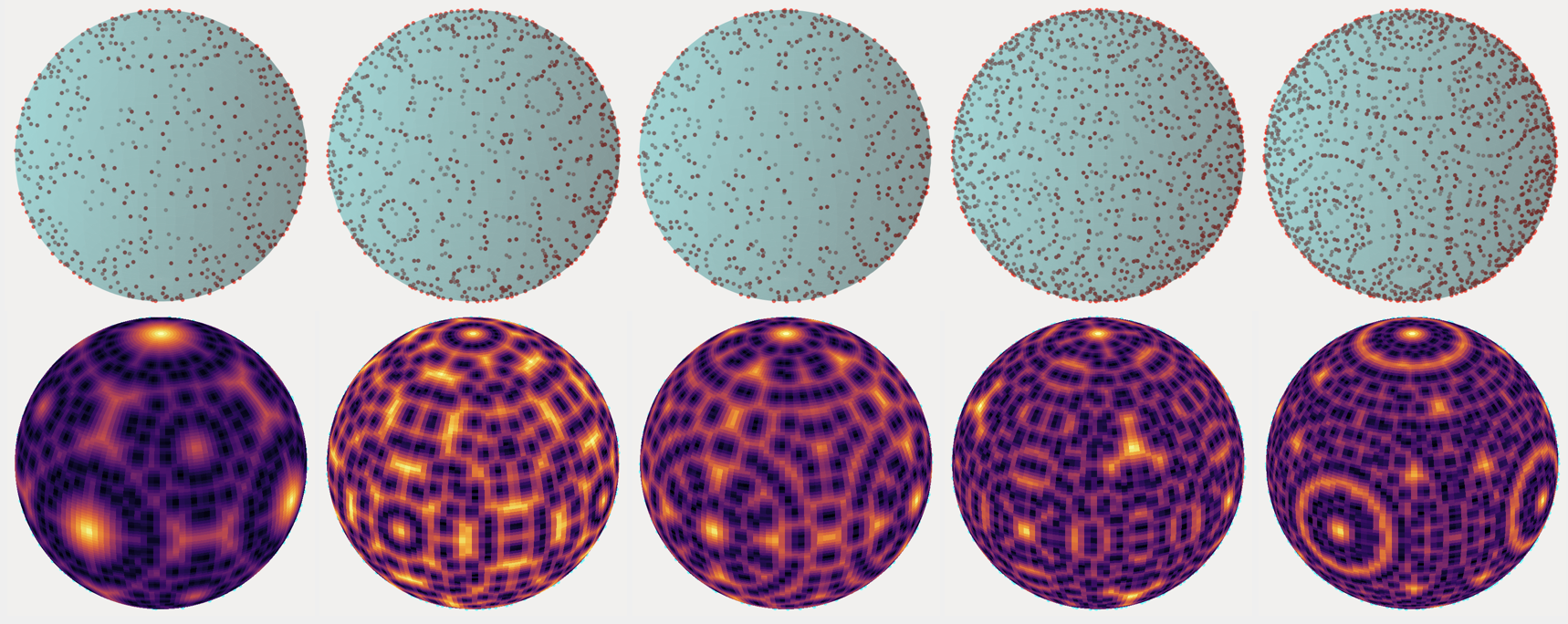}
  \caption{$\Omega_n\subset S^2$ for $n=89,93,101,161,171$, together with heatmaps of the distance function to $\Omega_n$. Brighter regions indicate points on $S^2$ that are farther from $\Omega_n$.}
  \label{fig:Omega_n holes}
\end{figure}


In parallel to \cref{thm:BRS all levels}, we also upgrade the conditional result of Humphries and Radziwi\l\l\, to all admissible levels.

\begin{thm}\label{thm:conditional Linnik}
Assume the Lindel\"of Hypothesis for standard $\GL(2)/\Q$ $L$-functions. Let $d\not\equiv 7$ (mod 8)
be a positive squarefree integer, let $\ell$ be odd, and set $n=d\ell^2$. Then the Diophantine equation $x^2+y^2+z^2=n$ has a primitive solution in the integers with $|z|=O(n^{1/3+\eps})$.
\end{thm}

\section{Spherical Caps and Zonal Harmonics}

The surface area of the $d$-dimensional unit sphere $S^d$ is $$\omega_d = \frac{2 \pi^{\frac{d+1}{2}}}{\Gamma\left(\frac{d+1}{2}\right)}.$$
We denote by $\mu$ the normalized rotation-invariant measure on the sphere.

\subsection{Spherical caps}

Fix a spherical cap $C_R(\alpha )$, centered at $\alpha\in S^d$ and having (Euclidean) radius $R$, i.e., $C_R(\alpha):=\{x\in S^d: |x-\alpha|<R\}$.

\begin{lm}\label{lm:area cap}
For any $\alpha\in S^d$ and $R\ll 1$, the normalized surface area of the spherical cap $C_R(\alpha)$ is given by
\begin{align}\label{eq:measure spherical cap}
\mu(C_R) = \frac{\omega_{d-1}}{\omega_d} \frac{R^d}{d} + O_d(R^{d+2}).
\end{align}
\end{lm}

\begin{proof}
    Choosing a local system of spherical coordinates about $\alpha$ on $S^d$,  let $\theta =\theta(x)\in[0,\pi )$ be the angle subtended at the origin of the vectors $x$ and $\alpha$ so that  $\langle x,\alpha \rangle = \cos \theta$.  We have $x\in C_R(\alpha )$ if and only if $\cos \theta \in(1-\tfrac{R^2}{2},1]$. By the half-angle formula, this is equivalent to $\theta\in[0,r)$ for $r:=2\arcsin(R/2)$. Taylor expansion about $R=0$ yields $r=R+O(R^3)$. Then
    \begin{align*}
\mu(C_R) &= \frac{\omega _{d-1}}{\omega _d} \int_0^r \sin(\theta )^{d-1}\, d \theta =\frac{\omega _{d-1}}{\omega _d} \frac{r^d}{d} + O(r^{d+2}). 
\end{align*}
\end{proof}
We will later work with the normalized characteristic function given by 
\begin{align}\label{Choose f}
  f_{R,\ga}(x) \coloneqq \begin{dcases}
    \frac{1}{\mu(C_R)}&\text{ if } x\in C_R(\alpha ) ;\\
    0 &\text{ otherwise}.
  \end{dcases}
\end{align}
This function is zonal, meaning that it depends only on $\scal{x,\alpha }$. As such it admits the (formal) spectral expansion
\begin{align} \label{spectral_expansion_general}
  f_{R,\alpha } (x) = \sum_{\nu\geq0} \hat f_{R,\alpha }(\nu) Z_\nu(x,\alpha),
\end{align}
in terms of the Gegenbauer coefficients $\hat f_{R,\alpha }(\nu)$ and the zonal harmonics $Z_\nu(x,\alpha)$ with pole $\alpha$. We review properties and estimates of these functions in the remainder of this section.
\subsection{Zonal harmonics}\label{sec:zonal harmonics}

Let $d\geq2$. The space $L^2(S^d)$, equipped with the standard inner product for the normalized rotation-invariant measure $\mu$  on the sphere, admits the direct sum decomposition
$$
L^2(S^{d}) = \bigoplus_{\nu\geq0} \cH_\nu,
$$
where $\cH_\nu$ is the eigenspace of the spherical Laplacian corresponding to the eigenvalue $-\nu(\nu+d-1)$. Each space $\cH_\nu$ is finite-dimensional, with
\begin{align*}
  {\rm dim}(\cH_\nu) = \binom{d+\nu }{\nu } -\binom{d+\nu -2}{\nu -2} \asymp \nu ^{d-1}
\end{align*}
if $\nu \geq2$, $\dim \cH_1 =d+1$, $\dim \cH_0=1$, and consists of projections to $S^d$ of real homogeneous harmonic polynomials of degree $\nu $ in $d+1$ variables. We will denote by $P$ both the polynomial and its projection to $\cH_\nu$, where the distinction should be clear from context. For each $x\in S^d$, there exists a unique real-valued function $Z_\nu(x,\cdot)$ such that 
\begin{align*}
 P(x) = \int_{S^d} P(y) Z_\nu(x,y) \, d \mu (y) 
\end{align*}
for each $P\in \cH_\nu$; the function $Z_\nu (x,y)$ is called the zonal harmonic of degree $\nu$ with pole $x$.  For any orthonormal basis $\cB_\nu\subset \cH_\nu$ of real-valued spherical harmonics, one has   
\begin{align}\label{eq:pre-trace formula}
  Z_\nu (x,y) = \sum_{P\in \cB_\nu } P(x)P(y),
\end{align}
and  this expression  does not depend on the particular choice of basis. In particular we take note of the identities
\begin{align*}
  &\int_{S^d} Z_\nu (x,z) Z_\nu (y,z) d \mu (z) =
  Z_\nu (x,y); \\
  &\int_{S^d} Z_\nu(x,x)\, d \mu (x) = 
  \dim \cH_\nu.
\end{align*} 

\subsection{Quaternionic Hecke operators on $S^2$}

The two-dimensional sphere is special in that it carries an additional
arithmetic symmetry coming from the Hamilton quaternion algebra.  This symmetry gives rise to a commutative algebra of Hecke operators on $L^2(S^2)$ which commute with the spherical Laplacian. We now recall the construction.

Let $B$ denote the Hamilton quaternion algebra over $\Q$, generated by the basis $\{1,i,j,k\}$ such that $i^2 = j^2 =-1$ and $ij=-ji =k$. The conjugate of a quaternion $x=a+bi+cj+dk$ is $\overline{x}=a-bi-cj-dk$. We define the reduced trace $\mathrm{tr}(\alpha) : B \rightarrow \Q$ by $\mathrm{tr}(x) = x + \overline{x}$ and the reduced norm $\mathrm{nr}(\alpha) : B \rightarrow \Q $ by $\mathrm{nr}(x) = x \overline{x}$. The Hurwitz order $\mathcal{O} \subset B$ is given by
$$
\mathcal{O} = \{(a + bi + cj + dk)/2 : (a,b,c,d) \in \Z^4 \text{ with the same parity} \}.
$$
It has class number 1, and its unit group has order $24$. Let $B_\infty=B\otimes_\Q \R$ be the classical Hamilton quaternions.  We write $B_\infty^\times = \{x \in B_\infty: \mathrm{nr}(x) \neq 0 \}$ for its group of invertible elements. The 2-sphere $S^2$ can be identified with the set of pure unit quaternions 
$$
\{ z \in  B_\infty : \mathrm{tr}(z)=0,\, \mathrm{nr}(z) = 1 \} \subset B_\infty^\times.
$$
There is a natural action of $B_\infty^\times$ on $S^2$ given by conjugation, i.e., $x\mapsto zxz^{-1}=(z x\overline{z})/\mathrm{nr}(z)$, which describes rotations. For $p \neq 2$ prime we have the Hecke operator $\tilde{T}_p$ on $L^2(S^2)$ given by 
\begin{align} \label{Heckeoperators:Brandt_matrices}
    (\tilde{T}_p f)(x) &= \frac{1}{|\mathcal{O}^\times|} \sum_{\substack{z \in \mathcal{O} \\ \mathrm{nr}(z) = p }} f \left( \frac{z x \overline{z} }{\mathrm{nr}(z)} \right).
\end{align}
For $p=2$ one can also define an involution operator $\tilde{W}_2$, see \cite{bocherer1994mellin}. The above operators commute, and also commute with the spherical Laplacian, so there exists a basis $\mathcal{B}_\nu$ for $\mathcal{H}_\nu(S^2)$ consisting of simultaneous eigenfunctions.

\subsection{Gegenbauer polynomials}
Zonal harmonics are the spherical harmonic analysis analogues of radial functions in Euclidean harmonic analysis. They have a particularly simple expression in terms of the Gegenbauer polynomials $C_\nu^\gl$, which can be defined via the generating function
\begin{align*}
  (1-2rt+r^2)^{-\lambda } = \sum_{\nu=0}^\infty C_\nu^\gl(t) r^\nu,
\end{align*}
where $|r|<1$, $|t|\leq 1$, and $\lambda>0$.  We fix $\lambda \coloneqq \tfrac{d-1}{2}$,  $C_\nu \coloneqq C_\nu ^\lambda $, and $c_\nu=\frac{\nu+\lambda}{\lambda}$; then 
\begin{align*}
  Z_{\nu }(x,y) =c_{\nu}\cdot C_\nu(\scal{x,y}).
\end{align*}
The Gegenbauer polynomials $C_\nu$ form an orthogonal basis for the space $\cH$ of $L^2$-functions on $[-1,1]$ with respect to the measure $(1-t^2)^{d/2-1}dt$. An $L^2$-function $f$ on the sphere $S^d$ that is zonal, i.e., that only depends on $\scal{x,\ga}$ for some fixed $\ga \in S^d$, may be seen as an element of $\cH$. Then the Funk--Hecke formula \cite[Theorem 1.2.9]{DaiXu2013} implies that $f$ admits the spectral expansion
\begin{align*} 
  f(x) =  \sum_{\nu\geq0} \hat f(\nu ) Z_\nu (x,\alpha )
\end{align*}
(with convergence in $L^2(S^d)$), where $\hat f(\nu)$ are the Gegenbauer coefficients
\begin{align} \label{eq:Gegenbauercoeff}
  \widehat{f}(\nu ) = \frac{\omega_{d-1}}{\omega_d} \frac{c_{\nu}}{\dim(\cH_\nu )}\int_{-1}^1 f(t) C_\nu(t) (1-t^2)^{d/2-1}dt.
\end{align}

Recall the characteristic function $f_{R,\alpha }$ defined in \cref{Choose f}. We have the following result.
  \begin{prop}\label{prop:Bessel coeffs}
We have $\hat{f}_{R,\ga}(0)=1$ and for $\nu>0$ and $R\ll 1$
    \begin{align*}
      |\hat f_{R,\ga}(\nu )| \ll_d \begin{dcases}
        1 &\text{ if } r\leq \tau ;\\
        (\nu R)^{-(d+1)/2} &\text{ if } r> \tau;
      \end{dcases}
    \end{align*}
where $\tau = \frac{c_d}{\nu + c_d'}$ with $c_d, c_d'$ two constants that depend on $d$ only, and $r=2\arcsin(R/2)$.
  \end{prop}

   \begin{proof}
The first statement follows from a simple change of variables
\begin{align*}
  \hat f_{R,\alpha }(0) = \frac{\omega _{d-1}}{\omega _d} \int_{0}^\pi  f_{R,\alpha }(\cos \theta ) (\sin \theta )^{d-1} d \theta =1.
\end{align*}
For $\nu >0$, the $\nu$-th Gegenbauer coefficient of $f_{R,\alpha }$ is given by
\begin{align*}
\hat f_{R,\ga}(\nu ) &= \frac{\xi_d}{\mu(C_R)} \frac{c_{\nu}}{\dim \cH_\nu } \int_0^r C_\nu (
  \cos \theta)(\sin \theta )^{d-1}\, d \theta
\end{align*}
for some normalizing constant $\xi_d$; the constant $\xi_d$ will come to change in the course of the proof, the only important point is that it depends on $d$ only.  Let $\lambda=\tfrac{d-1}{2}$ as earlier.  The extension of the Hilb formula to Jacobi polynomials (see \cite[Thm 8.21.12]{Szego1939}  specialized to $\alpha =\beta =\tfrac{d}{2}-1$ and $P_\nu (t)=\frac{\Gamma (\nu +\lambda +1/2)}{\Gamma (\nu +2 \lambda )}C_\nu (t)$) states that
\begin{align*}
  C_\nu (\cos \theta ) = \xi_d \frac{(\nu +\lambda )^{1-d/2}\Gamma (\nu +2\lambda )\theta ^{1/2}}{\nu ! (\sin \theta )^{(d-1)/2}} J_{d/2-1}((\nu +\lambda )\theta  ) + E(\nu,\theta )
\end{align*}
where $J$ is the standard $J$-Bessel function and the error term is
\begin{align*}
  E(\nu ,\theta )\ll \begin{dcases}
   \theta^{1/2} \nu^{-3/2}, &\text{ if } \theta \geq \tfrac{c}{\nu } ;\\
    \theta^{d/2+1}\nu^{d/2-1}, &\text{ if } \theta <\tfrac{c}{\nu },
   \end{dcases}
  \end{align*}  
for some fixed constant $c$.  The main term contributes
\begin{align*}
  \int_0^r \theta^{d/2} J_{d/2-1}((\nu +\lambda) \theta ) d \theta  = (\nu+\lambda )^{-1-d/2} \int_0^{(\nu+\lambda ) r} \theta^{d/2} J_{d/2-1}(\theta ) d \theta  .
\end{align*}
To go further, we require some asymptotic bounds for the $J$-Bessel function. If $\theta >c_d$, then \cite[Eq.~(8.451.1)]{GradshetynRyzhik1996} provides the asymptotic bound
         \begin{align*}
       J_{d/2-1}(\theta ) = 
       \sqrt{\frac{2}{\pi \theta }} \left(\cos(\theta -\tfrac{\pi}{4}(d-1))+O\left(\theta  ^{-1}\right) \right) 
     \end{align*}
for some constant depending on $d$ only. If $\theta  \leq c_d$, then we instead rely on the rough bound $|J_{d/2-1}(\theta )|=O_d(1)$, derived from \cite[Eq.~(8.440)]{GradshetynRyzhik1996}. Hence if $r> c_d/(\nu+\lambda)=\tau$, we have 
\begin{align*}
  \hat f_{R,\alpha }(\nu ) \ll (\nu R)^{-d} \frac{\nu^{2-d} (\nu +d-2)!}{\nu !}  ((\nu R)^{(d-1)/2}+1) \asymp (\nu R)^{-(d+1)/2},
\end{align*}
where we use that $r\asymp R$. The case $r\leq \tau$ is similar.
\end{proof}

\subsection{The 1-Wasserstein distance}

Let $d_{S^2}$ denote the geodesic distance on $S^2$. For two Borel
probability measures $\mu_1,\mu_2$ on $S^2$, their 1-Wasserstein
distance is
$$
W_1(\mu_1,\mu_2)
\coloneqq
\inf_{\pi\in\Pi(\mu_1,\mu_2)}
\int_{S^2\times S^2} d_{S^2}(x,y)\,d\pi(x,y),
$$
where $\Pi(\mu_1,\mu_2)$ denotes the set of couplings of $\mu_1$ and
$\mu_2$. Equivalently, by Kantorovich--Rubinstein duality \cite{Villani2003},
$$
W_1(\mu_1,\mu_2)
=
\sup_{{\rm Lip}(f)\leq 1}
\left|
\int_{S^2}f\,d\mu_1-\int_{S^2}f\,d\mu_2
\right|.
$$
Here ${\rm Lip}(f)$ is computed with respect to the spherical distance $d_{S^2}$.
Using the Euclidean distance instead changes $W_1$ only by an absolute multiplicative constant.

We will use the following elementary spectral smoothing inequality. It is a heat-kernel version of the Berry--Esseen inequalities for compact spaces; see also \cite{BordaCuenin2025,SanchezGarza2026} for related formulations on spheres and compact
homogeneous spaces.

\begin{prop}\label{prop:wasserstein-smoothing}
Let $\eta$ be a Borel probability measure on $S^2$. For each
$\nu\geq 0$, let $\mathcal B_\nu$ be an orthonormal basis of
$H_\nu(S^2)$, and write
$$
\widehat \eta(P):=\int_{S^2}P\,d\eta.
$$
Then, for $0<t\leq1$,
$$
W_1(\eta,\mu)
\ll
t^{1/2}
+
\left(
\sum_{\nu\geq 1}
\frac{e^{-2t\nu(\nu+1)}}{\nu(\nu+1)}
\sum_{P\in\mathcal B_\nu}
|\widehat\eta(P)|^2
\right)^{1/2}.
$$
\end{prop}

\begin{proof}
Let $\sigma=\eta-\mu$ and let \(P_t=e^{t\Delta}\) be the heat semigroup on $S^2$. We need to bound $\left| \int f d\sigma \right|$ uniformly over $1$-Lipschitz functions $f$. The standard heat-kernel estimate 
$$
\|f-P_tf\|_\infty\ll t^{1/2}
$$
gives 
$$
\left|\int_{S^2} f\,d\sigma\right|
\leq
\left|\int_{S^2} P_t f\,d\sigma\right|+O(t^{1/2}).
$$
Since $\sigma(S^2)=0$, the degree zero term does not contribute. Taking the spectral expansion of
$f$ and using that
$$
P_tP=e^{-t\nu(\nu+1)}P
\qquad (P\in H_\nu),
$$
we get
$$
\int_{S^2} P_tf\,d\sigma
=
\sum_{\nu\geq 1}
e^{-t\nu(\nu+1)}
\sum_{P\in\mathcal B_\nu}
\widehat f(P)\widehat\sigma(P).
$$
By Cauchy--Schwarz,
\[
\left|\int P_tf\,d\sigma\right|
\leq
\left(
\sum_{\nu\geq 1}\nu(\nu+1)
\sum_{P\in\mathcal B_\nu}|\widehat f(P)|^2
\right)^{1/2}
\left(
\sum_{\nu\geq 1}
\frac{e^{-2t\nu(\nu+1)}}{\nu(\nu+1)}
\sum_{P\in\mathcal B_\nu}|\widehat\sigma(P)|^2
\right)^{1/2}.
\]
The first factor is \(\|\nabla f\|_2\leq \operatorname{Lip}(f)\leq 1\).
For \(\nu\geq 1\), we have \(\widehat\sigma(P)=\widehat\eta(P)\), since
\(\int_{S^2}P\,d\mu=0\). Taking the supremum over all \(1\)-Lipschitz
functions \(f\) gives the claim.
\end{proof}

\section{Fourier Coefficients of Theta Series}

\subsection{Classical material}\label{sec:classical stuff}
The set $\Omega_n$ of all rational points on the sphere $S^d$ of height $n$ is parametrized by
\begin{align*}
  \Omega_n = \left\{ n^{-1}(m_1,\dots,m_{d+1}) \in n^{-1}\Z^{d+1} \mid \begin{array}{c}  m_1^2+\dots + m_{d+1}^2=n^2\\  (m_1,\dots,m_{d+1},n)=1 \end{array}\right\}.
\end{align*}
It is closely connected to the sum-of-squares functions
\begin{align*}
  r_{d+1}(n) &= \left\{ (m_1,\dots,m_{d+1})\in \Z^{d+1}\mid  m_1^2+\dots + m_{d+1}^2=n\right\},\\
  r^*_{d+1}(n) &=  \left\{ (m_1,\dots,m_{d+1})\in \Z^{d+1}\mid \begin{array}{c}  m_1^2+\dots + m_{d+1}^2=n\\  (m_1,\dots,m_{d+1},n)=1 \end{array}\right\}.
\end{align*}
In fact we have 
\begin{align}\label{eq:Omega_n}
    |\Omega_n| = r^*_{d+1}(n^2) = \sum_{\delta \mid n} \mu (\delta ) r_{d+1}(n^2/\delta ^2),
\end{align}
where $\mu$ is the M\"obius function. For small dimensions we have access to exact formulas, e.g., 
 \begin{align*}
   r^*_3(n^2)= \begin{dcases} 6n\prod_{p\mid n} (1-\chi_{-4}(p)p^{-1}) &\text{if } n \text{ is odd};\\
     0 &\text{if } n \text{ is even}.
   \end{dcases}
 \end{align*}
For $d=3$, Jacobi's four-square formula and M\"obius inversion similarly give, for odd $n$, 
$$
r_4^*(n^2) = 8n^2 \prod_{\substack{p\mid n\\ p \text{ odd}}} (1+p^{-1}).
$$
More generally, we have $r^*_{d+1}(n^2) \asymp n^{d-1}$  for $n$ in a subset $\mathcal N_d \subset \N$ of positive density;   see  \cref{app:size of Omega_n}. The sum-of-squares functions appear as the Fourier coefficients of the standard theta-functions
\begin{align*}
  \Theta(\tau) = \sum_{m\in \Z^{d+1}}  q^{|m|^2} = \sum_{n\geq0} r_{d+1}(n)q^n
\end{align*}
for $\tau\in \h$ and $q=e^{2\pi i\tau}$, which is a modular form of weight $(d+1)/2$ and level 4.

\subsection{Modular forms input}
Let $P$ be a real homogeneous harmonic polynomial of degree $\nu>0$ in $d+1$ variables. Its associated theta function evaluated on the lattice $\Z^{d+1}$ is defined by 
\begin{align*}
  \Theta_P(\tau) = \sum_{m\in \Z^{d+1}} P(m) q^{|m|^2}
\end{align*}
for $\tau\in \h$ and $q=e^{2\pi i\tau}$. The theta series $\Theta_P$ is a cusp form of weight $\tfrac{d+1}{2}+\nu$ and level 4 \cite{Schoeneberg1939,Pfetzer1953,Shimura1973} and its  Fourier coefficients are
\begin{align*}
r_P(n) \coloneqq  \sum_{\substack{m\in \Z^{d+1}\\ |m|^2=n}} P(m).
\end{align*}
We may assume that $\nu$ is even since $r_P(n)=0$ otherwise.

\begin{lm}\label{lm:norm-bound}
Let $P$ be a real homogeneous harmonic polynomial of degree $\nu>0$ in $d+1$ variables. If $d$ is odd, set $F_P:= \Theta_P$ and $k:=\tfrac{d+1}{2}+\nu$.  If $d$ is even, let $F_P := \cS(\Theta_P)$ be the Shimura lift of $\Theta_P$ of weight $k:=d+2\nu$. Then
$$
\|F_P\|^2 \ll_d \|P\|_\infty^2 \frac{\Gamma(k)}{(4\pi)^k},
$$
uniformly in $\nu$.  
\end{lm}

\begin{proof}
For a positive integer $M$, let $E_M(z,s)$ be the non-holomorphic Eisenstein series for $\Gamma_0(M)$ at the cusp $\infty$, defined for $\Re(s)>1$ by
$$
E_M(z,s) = \sum_{\gamma \in \Gamma_\infty \backslash \Gamma_0(M)} \Im(\gamma z)^s.
$$
It extends meromorphically to $s \in \C$ and has a simple pole at $s=1$ with positive residue. \newline
Assume first that $d$ is odd and set $k_0:=k-\nu=\tfrac{d+1}{2}$. By Rankin--Selberg unfolding, we find that for $\Re(s)>1$ 
\begin{align*}
I_P(s):=\int_{\Gamma_0(4) \backslash \h} |\Theta_P|^2 E_4(z,s) y^{k-2} dxdy &= \frac{\Gamma(s+k-1)}{(4\pi)^{k+s-1}} \sum_{n\ge 1} \frac{|r_P(n)|^2}{n^{s+k-1}}.
\end{align*}
Since
$$
|r_P(n)| \leq n^{\nu/2} \sum_{|m|^2=n} |P(\tfrac{m}{\sqrt n})| \leq  \|P\|_\infty n^{\nu/2} r_{d+1}(n),
$$
we obtain for real $s>1$
\begin{align}\label{eq:IP(s)_bound}
I_P(s) \leq \|P\|_\infty^2  \frac{\Gamma(s+k-1)}{(4\pi)^{k+s-1}} D_d(s),
\end{align}
where $D_d(s) := \sum_{n\ge 1} \frac{|r_{d+1}(n)|^2}{n^{s+k_0-1}}$. The theta series $\Theta(z)=\sum_{n\ge0} r_{d+1}(n)q^n$
is a (non-cuspidal) modular form of weight $k_0$ and level $4$.  Fix a prime $p$ and consider
$$
J_p(s) := \int_{\Gamma_0(4p^2) \backslash \h} \big( |\Theta(z)|^2 -|\Theta(p^2 z)|^2 \big) E_{4p^2}(z,s) y^{k_0-2} dxdy.
$$
Since $\Theta$ has constant term $1$ at $\infty$, the integral converges for $\Re(s)>1$ and unfolds in the usual way, and we obtain 
$$
J_p(s) = \frac{\Gamma(s+k_0-1)}{(4\pi)^{s+k_0-1}} \big( 1-p^{-2(s+k_0-1)} \big) D_d(s).
$$
From this we see that $D_d(s)$ extends meromorphically to a neighborhood of $s=1$ with at most a simple pole there. The claim follows by taking the residue at $s=1$ on both sides of \eqref{eq:IP(s)_bound}. \\
The case where $d$ is even is similar after an application of Shimura's correspondence. Set $k_0 := k- 2\nu = d$. Let $\cS(\Theta)$ and $\cS(\Theta_P)$ be the Shimura lifts of $\Theta$ and of $\Theta_P$,  respectively; see \cite[Def 15.1.20 for $D=1$]{CohenStromberg}. Then $\cS(\Theta_P)$ is a cusp form of integral weight $k=d+2\nu$ and level $2$, with Fourier coefficients
\begin{align}\label{eq:Shimura coeffs}
a_{\cS(\Theta_P)}(n) = \sum_{\delta\mid n} \delta^{k/2-1} r_P(n^2/\delta^2).
\end{align}
As in the odd-dimensional case, its Petersson norm can be estimated by comparison with $\cS(\Theta)$. Indeed, using the trivial bound
$$
|a_{\cS(\Theta_P)}(n)| \leq  \|P\|_\infty n^\nu \sum_{\delta\mid n} \delta^{k/2-1-\nu} r_{d+1}(n^2/\delta ^2) = \|P\|_\infty n^\nu a_{\cS(\Theta )}(n),
$$
a similar Rankin--Selberg argument as above yields
$$
\|\cS(\Theta_P)\|^2 \ll_d \|P\|_\infty^2 \frac{\Gamma(k)}{(4\pi)^k}.
$$
\end{proof}

\begin{prop}\label{prop:bounds on coeffs dim d}
For every $\eps>0$ and every real homogeneous harmonic polynomial $P$ of degree $\nu>0$ we have
\begin{align*}
|r_P(n^2)| \ll_{d,\eps} \|P\|_\infty \nu^{1/2+\eps} n^{(d-1)/2+\nu+\eps}.
\end{align*}
\end{prop}

\begin{proof}

Assume first that $d$ is odd and set $k = \tfrac{d+1}{2}+\nu$. Then $\Theta_P$ is a cusp form of integral weight $k$ and level $4$. Let $\{f_i\}$ be an orthogonal basis of normalized Hecke eigenforms spanning $S_k(4,\chi)$. Then
$
r_P(n^2) = \sum_i \frac{\scal{\Theta _P,f_i}}{\|f_i\|^2} a_{f_i}(n^2)
$
and by Cauchy-Schwarz and Parseval's identity we find that 
$$
|r_P(n^2)|^2 \le \|\Theta _P\|^2 \sum_i \frac{|a_{f_i}(n^2)|^2}{\|f_i\|^2}. 
$$
Applying Deligne's bound $|a_{f_i}(n^2)| \ll_\eps n^{k-1+\eps}$ and $\|f_i\|^2 \gg_\eps k^{-\eps} \frac{\Gamma(k)}{(4\pi)^k}$ from \cite[Thm 4]{Fomenko1991} gives 
$$
\sum_i \frac{|a_{f_i}(n^2)|^2}{\|f_i\|^2} \ll_{d, \eps} \nu^{1+\eps}  n^{2k-2+2\eps} \frac{(4\pi)^k}{\Gamma(k)},
$$
where we used that $\dim(S_k(4,\chi ))=O(k)$. We estimate the Petersson norm using Lemma~\ref{lm:norm-bound} and conclude.  \\
Now assume that $d$ is even and set $k=d+2\nu$. Then $\cS(\Theta_P)$ is a cusp form of integral weight $k$ and level $2$, with Fourier coefficients given by \eqref{eq:Shimura coeffs}. Applying the same argument as above to $\cS(\Theta_P)$, and using Lemma~\ref{lm:norm-bound}, we obtain
$$
a_{\cS(\Theta_P)}(n) \ll_{d,\eps} \|P\|_\infty \nu^{1/2+\eps} n^{(k-1)/2+\eps}.
$$
M\"obius inversion in the divisor-sum identity \eqref{eq:Shimura coeffs} then yields
$$
r_P(n^2) \ll_{d,\eps} \|P\|_\infty \nu^{1/2+\eps} n^{(k-1)/2+\eps} = \|P\|_\infty \nu^{1/2+\eps} n^{(d-1)/2+\nu+\eps}.
$$
\end{proof}

\subsection{Theta series of lattices over quaternion algebras} \label{subsection:thetaseries_quaternion}
In the proof of  \cref{prop:bounds on coeffs dim d} we had to expand the theta series $\Theta_P$ in a linear combination of Hecke eigenforms. This expansion leads to a bound on the Fourier coefficients $r_{P}(n^2)$ that depends on the dimension of the space of cusp forms of weight $\tfrac{d+1}{2}+\nu$, which adversely affects the rate of small-scale equidistribution for $\Omega_n$. In this subsection, we focus specifically on the two-sphere $S^2$, where we can bypass the expansion by using maximal orders in quaternion algebras to work directly with theta series that are already Hecke eigenforms. 

Let $\Lambda$ be the lattice of the trace zero elements in the suborder $\Z + 2 \mathcal{O}$:
$$
 \Lambda = \{bi + cj + dk : (b,c,d) \in \Z^3 \text{ with the same parity} \}.
$$
We consider the quaternionic theta series (the `Waldspurger lift' of $P$) 
$$
\Theta_{\Lambda,P} = \sum_{x\in \Lambda} P(x) q^{\mathrm{nr}(x)} =  \sum_{n=1}^\infty r_{\Lambda,P}(n)  q^n.
$$

The parity condition defining $\Lambda$ forces $r_{\Lambda,P}(n)=0$ whenever $n\equiv1,2$ (mod 4). The theta series $\Theta_{\Lambda,P}$ is a cusp form that lies in the Kohnen plus space $S_{3/2+\nu}^+(4,\chi)$.

\begin{prop}\label{prop:bounds on coeffs dim 2}
Let $\nu>0$ be even. Let $P\in \cH_\nu(S^2)$ be a common eigenfunction of all Hecke operators $\tilde{T}_p$  (for odd primes $p$) and of the involution $\tilde{W}_2$. Then for every odd $n$ and every fundamental discriminant $D<0$ we have
$$
|r_{\Lambda,P}(|D|n^2)| \leq |r_{\Lambda,P}(|D|)|\, n^{\nu+1/2+o(1)}.
$$
\end{prop}

\begin{proof}
The action of the (half-integral weight) Hecke operators $T_{p^2}$ on the theta series $\Theta_{\Lambda,P}$ can be represented by the Hecke operators $\tilde{T}_p$ on $\mathcal H_\nu(S^2)$ defined in \cref{Heckeoperators:Brandt_matrices}. More precisely, we have the Eichler commutation relation 
\begin{align} \label{Eichler_commutation_relation}
    T_{p^2} \Theta_{\Lambda,P} = p^\nu \Theta_{\Lambda, \tilde{T}_p P};
\end{align}
see \cite{bocherer1994mellin} or \cref{appendix:Eichler}, where we give a self-contained proof. In particular, if $P$ is a common eigenfunction of all $\tilde T_p$ (for odd primes $p$), then $\Theta_{\Lambda,P}$ is an eigenform for all odd Hecke operators $T_{p^2}$. Furthermore, the operators $\tilde{T}_p$ satisfy the same multiplicative identities as the Hecke operators $T_{p^2}$. From our assumptions and the Eichler commutation relation \eqref{Eichler_commutation_relation} it follows that $\Theta_{\Lambda,P}  \in S_{3/2+\nu}^+(4,\chi)$ is a Hecke eigenform.  
By \cite[Thm 1]{kohnen1980modular}, the Shimura lift $\mathcal S_D(\Theta_{\Lambda,P}) = \sum_{n \ge 1} A_{D}(n)q^n$ corresponding to the fundamental discriminant $D$ is a Hecke eigenform of integral weight $2+2\nu$ and level 1 with Fourier coefficients
\begin{align}\label{def:Shim}
A_D(n) = \sum_{\delta \mid n} \chi_D(\delta) \delta^\nu r_{\Lambda,P} (|D|n^2/\delta^2).
\end{align}
If $A_D(1) = r_{\Lambda,P}(|D|)=0$, then $\mathcal S_D(\Theta_{\Lambda,P})$ is identically zero, and Möbius inversion gives $r_{\Lambda,P}(|D|n^2)=0$ for all $n$. Thus the claim is trivial in this case. We will henceforth assume $A_D(1) \neq 0$ and set $R_D(n) = A_D(n) / A_D(1)$. Deligne's bound applied to the normalized Hecke eigenform implies that $|R_{D}(n)|\leq d(n)n^{1/2+\nu}.$  Möbius inversion in \eqref{def:Shim} gives the identity
\begin{align*}
  r_{\Lambda,P}(|D|n^2) = r_{\Lambda,P}(|D|) \sum_{\delta\mid n} \mu(\delta)\chi_{D}(\delta) \delta^\nu R_{D}(n/\delta),
\end{align*}
from which the bound follows immediately. 

\end{proof}

\begin{rmk}
    To compare this bound with that of \cref{prop:bounds on coeffs dim d}, note that $|r_{\Lambda,P}(|D|)\leq \|P\|_\infty |D|^{\nu/2} r_{\Lambda,1}(|D|)$. We will later need the sharper form stated in \cref{prop:bounds on coeffs dim 2}.
\end{rmk}

\begin{coro}\label{coro:r_P(dn^2)}
For $P$ as in \cref{prop:bounds on coeffs dim 2}, $n$ odd and $d$ squarefree, we have   
\begin{align*}
|r_P(dn^2)|\ll |r_P(d)| n^{\nu+1/2+o(1)}.
\end{align*}
\end{coro}
\begin{proof}

Set
\begin{align*}
D_d = \begin{dcases}
    -4d, &\text{ if } d\equiv 1,2\, \text{(mod 4)} ;\\
    -d, &\text{ if } d\equiv 3\, \text{(mod 4)}.
  \end{dcases}
\end{align*}
Then $D_d$ is a fundamental discriminant. If $d\equiv 3$ (mod 4), then every integer solution of $x^2+y^2+z^2=dn^2$ has $x,y,z$ all odd, since $n$ is odd. Hence we have $r_{\Lambda,P}(dn^2)=r_P(dn^2)$. If $d\equiv 1,2$ (mod 4), then every solution of $x^2+y^2+z^2=4dn^2$ has $x,y,z$ all even; we conclude that $r_{\Lambda,P}(4dn^2)= 2^\nu r_P(dn^2)$. Thus, in both cases, \cref{prop:bounds on coeffs dim 2} applied with $D=D_d$ gives the result.

\end{proof}

\begin{rmk}
    By Legendre's three square theorem, we have $r_P(dn^2)=0$ if $d\equiv 7$ (mod 8).
\end{rmk}

For $d$ squarefree, Waldspurger's formula gives \cite[(5.4)]{bourgain2016spatial} 
\begin{align*}
    |r_P(d)|^2 = c\frac{d^{\nu+1/2}L(1/2,f)L(1/2,f\times \chi_{D_d})}{L(1,\mathrm{Sym}^2 f)}
\end{align*}
for an absolute constant $c>0$ and where $f$ a holomorphic Hecke cusp form of weight $2\nu+2$. With GLH and Hoffstein--Lockhart, this realization implies 
\begin{align*}
|r_P(d)|^2 \ll d^{\nu+1/2+\eps}\nu^\eps
\end{align*}
for $d\not\equiv 7$ (mod 8) squarefree. This implies that 
\begin{align} \label{eq:GLH:square-bound}
|r_P(dn^2)|\ll d^{\nu/2+1/4+\eps}\nu^\eps n^{\nu+1/2+o(1)}.
\end{align}

\section{From Weyl sums to small-scale equidistribution}

\subsection{Weyl sums and spectral inputs}

Let $P\in\mathcal H_\nu$ be a spherical harmonic of degree $\nu$.  We introduce the primitive Weyl sum 
$$
W_n(P) := \sum_{x\in \Omega_n} P(x) = n^{-\nu} r_P^*(n^2),
$$
for 
\begin{align} \label{eq:Moebiusinv_rPprim}
    r_P^*(n^2)  &\coloneqq \sum_{\substack{m\in\Z^{d+1}\\ |m|^2=n^2\\ (m,n)=1}} P(m) = \sum_{\delta\mid n} \mu(\delta) \delta^\nu  r_P(n^2/\delta^2),
\end{align}
where $r_P(m)$ is the $m$-th Fourier coefficient of the theta series $\Theta_P$.

By \cref{prop:bounds on coeffs dim d} we get, for any $\eps>0$,
\begin{align}\label{eq:NT input}
W_n(P) &\ll \|P\|_\infty \nu^{1/2 +\eps/2 } n^{(d-1)/2  + \eps/2 }  \sum_{\delta\mid n} \delta^{-(d-1)/2 - \epsilon/2}  \nonumber\\ 
  & \ll \|P\|_\infty \nu^{1/2+\eps/2}  n^{(d-1)/2  + \eps/2 }  d(n)  \nonumber \\
  &\ll \|P\|_2 \nu^{d/2+\varepsilon} n^{(d-1)/2  + \eps },
\end{align}
where in the last step we use the sup-norm bound $\|P\|_\infty \ll \|P\|_2 \nu^{(d-1)/2}$ for eigenfunctions of the spherical Laplacian \cite{avakumovic1956eigenfunktionen,levitan1952asymptotic}. For the zonal spherical harmonic $P = Z_\nu(\cdot,\alpha)$, we have $\|P\|^2_2= \dim \cH_\nu\asymp \nu^{d-1}$.   
Hence
\begin{align} \label{eq:WS_bound}
W_n(Z_\nu(\cdot,\alpha)) \ll \nu^{d-1/2+\eps} n^{(d-1)/2+\eps}. 
\end{align}

We now specialize to $S^2$. For each $\nu>0$, let $\cB_\nu$ be an orthonormal basis of $H_\nu(S^2)$ consisting of simultaneous eigenfunctions of the quaternionic Hecke operators $\tilde T_p$ for odd primes $p$, and of the involution $\tilde W_2$. Let $n$ be odd and $P\in \cB_\nu$. If $\nu$ is odd then $W_n(P)=0$ and $r_P(1)=0$ by the symmetry $m\mapsto -m$. Hence it suffices to consider even $\nu$, where \cref{coro:r_P(dn^2)} applies. Applying \cref{coro:r_P(dn^2)} with $d=1$ we obtain instead
\begin{align} \label{eq:Hecke_bound}
W_n(P) \ll |r_P(1)| n^{1/2+o(1)}.
\end{align}

We shall also use the pre-trace formula \cref{eq:pre-trace formula} for spherical harmonics. In the notation of \cref{sec:zonal harmonics}, for every orthonormal basis $\cB_\nu\subset H_\nu(S^2)$,
$$
Z_\nu(x,y)=\sum_{P\in\mathcal B_\nu}P(x)P(y).
$$
In particular,
$$
\sum_{P\in\cB_\nu}P(x)^2 =Z_\nu(x,x)= \dim H_\nu(S^2) =2\nu+1.
$$
Moreover, by Cauchy--Schwarz we get 
\begin{align} \label{eq:Zv_estimate}
|Z_\nu(x,y)| \leq Z_\nu(x,x)^{1/2}Z_\nu(y,y)^{1/2} = 2\nu+1.
\end{align}
We will repeatedly make use of the following estimates.

\begin{prop}\label{prop:spectral second moment}
Let $n$ be odd and let $\cB_\nu$ be an orthonormal Hecke eigenbasis of $H_\nu(S^2)$. Then 
$$
\sum_{P\in \cB_\nu} |W_n(P)|^2 \ll n^{1+o(1)}(2\nu+1)
$$
and uniformly in $\alpha\in S^2$,
$$
\sum_{P\in \cB_\nu} |P(\alpha)W_n(P)| \ll n^{1/2+o(1)}(2\nu+1).
$$
\end{prop}

\begin{proof}
The bound \eqref{eq:Hecke_bound} gives 
$$
\sum_{P \in \cB_\nu} |W_n(P)|^2 \ll n^{1+o(1)} \sum_{P \in \cB_\nu} |r_P(1)|^2.
$$
We note that
$$
r_P(1) = \sum_{x\in S} P(x), \qquad S=\{\pm e_1,\pm e_2,\pm e_3\}.
$$
Therefore
$$
\sum_{P\in\cB_\nu}|r_P(1)|^2 = \sum_{x,y\in S}Z_\nu(x,y) \ll 2\nu+1,
$$
where we used \eqref{eq:Zv_estimate}. For the second estimate, Cauchy--Schwarz gives 
\begin{align*}
\sum_{P\in \cB_\nu} \left|P(\alpha)W_n(P)\right|
& \le \left(\sum_{P\in\cB_\nu} |P(\alpha)|^2 \right)^{1/2}   \left( \sum_{P\in \cB_\nu} |W_n(P)|^2  \right)^{1/2} \\
&\ll (2\nu+1)^{1/2} \left( n^{1+o(1)}(2\nu+1) \right)^{1/2} \\
&= (2\nu+1) n^{1/2+o(1)}.
\end{align*}  
\end{proof}

\subsection{Smoothing} \label{sec:smoothing}
For each $\nu>0$, fix an orthonormal basis $\cB_\nu$ of $\cH_\nu$. The spectral expansion \cref{spectral_expansion_general} of $f_{R,\alpha }$ fails to converge pointwise, so we need to regularize.  Let $\rho\in(0, R/2 )$ be a small parameter. We use the basis $\cB_\nu$ to spectrally expand the point average of the convolution $f_{R\pm \rho ,\alpha }\ast f_{\rho,\alpha }$ as follows
\begin{align} \label{eq:smoothed-spectral-expansion}
  \sum_{x\in \Omega_n} (f_{R\pm \rho ,\alpha }\ast f_{\rho,\alpha }) (x) = |\Omega_n| + \sum_{\nu>0} \hat f_{R\pm \rho ,\alpha }(\nu) \hat f_{\rho,\alpha }(\nu) \sum_{P\in \cB_\nu} P(\alpha ) \sum_{x\in \Omega_n} P(x).
\end{align} 
Note that for two zonal functions with pole $\alpha$, we write $f\ast g$ for the spherical convolution normalized so that $\widehat{f\ast g}(\nu)=\widehat f(\nu) \widehat g(\nu)$. The advantage is that this smoothed expansion converges pointwise. We then rely on the sandwich inequality
\begin{align}\label{HR sandwich}
  \frac{\mu (C_{R-\rho})}{\mu (C_R)} (f_{R-\rho,\alpha }\ast f_{\rho,\alpha} ) \leq f_{R,\alpha} \leq \frac{\mu (C_{R+\rho})}{\mu (C_R)} (f_{R+\rho,\alpha }\ast f_{\rho,\alpha} )
\end{align}
as done in \cite{lubotzky1986hecke,HumphriesRadziwill2022}.

\subsection{Variance for square levels: proof of \cref{thm:variance} and \cref{coro:a.e. optimal}}
Let
$f_{R,\alpha} = \frac{\mathbf 1_{C_R(\alpha)}}{\mu(C_R)}.$
Then
$$
|\Omega_n\cap C_R(\alpha)| - |\Omega_n|\mu(C_R) = \mu(C_R) \left( \sum_{x\in\Omega_n}f_{R,\alpha}(x)-|\Omega_n| \right).
$$
Using the spectral expansion in $L^2$, and choosing $\mathcal B_\nu$ to be a Hecke eigenbasis, we get
$$
\sum_{x\in\Omega_n}f_{R,\alpha}(x)-|\Omega_n| = \sum_{\nu>0} \widehat f_R(\nu) \sum_{P\in\mathcal B_\nu} P(\alpha)W_n(P).
$$
By orthogonality and \cref{prop:spectral second moment} this is 
\begin{align*}
\int_{S^2} \left( |\Omega_n\cap C_R(\alpha)| - |\Omega_n|\mu(C_R) \right)^2 d\mu(\alpha) &= \mu(C_R)^2 \sum_{\nu>0} |\widehat f_R(\nu)|^2 \sum_{P\in\mathcal B_\nu} |W_n(P)|^2 \\
&\ll \mu(C_R)^2 n^{1+o(1)} \sum_{\nu>0} |\widehat f_R(\nu)|^2 (2\nu+1).
\end{align*}
By Parseval we have 
$$
\sum_{\nu\ge0} |\widehat f_R(\nu)|^2(2\nu+1) = \|f_{R,\alpha}\|_2^2 = \mu(C_R)^{-1}.
$$
Since $|\Omega_n|=n^{1+o(1)}$ for odd $n$, this gives
$$
\int_{S^2} \left( |\Omega_n\cap C_R(\alpha)| - |\Omega_n|\mu(C_R) \right)^2 d\mu(\alpha) \ll n^{o(1)}|\Omega_n|\mu(C_R).
$$
This proves \cref{thm:variance} and it remains to prove \cref{coro:a.e. optimal}. Chebyshev's inequality gives, for every $\lambda>0$,
$$
\mu\left( \left\{ \alpha\in S^2: \left| \frac{1}{|\Omega_n|}\sum_{x\in\Omega_n}f_{R,\alpha}(x)-1 \right|>\lambda \right\} \right) \ll \lambda^{-2} n^{o(1)} \bigl(|\Omega_n|\mu(C_R)\bigr)^{-1}.
$$
Fix $\delta>0$ and choose $\delta^-$ such that $\frac{\delta}{2}<\delta^-<\delta.$ For fixed $R$, set $\varepsilon_R := n^{\delta^-} \bigl( |\Omega_n|\mu(C_R) \bigr)^{-1/2}.$ Then Chebyshev gives an exceptional set $\cE_n(R)$ with measure $\mu(\cE_n(R)) \ll n^{-2\delta^- +o(1)}$ outside of which 
we have
\begin{align*}
  \frac{|\Omega_n \cap C_R(\alpha )|}{|\Omega_n| \mu (C_R)} = 1 + O(\epsilon_R ).
\end{align*}
If $R\gg n^{-1/2+\delta}$, then $|\Omega_n|\mu(C_R)\gg n^{2\delta+o(1)}$ and therefore $\eps_R=o(1)$. To make the exceptional set independent  of $R$, choose $0<\kappa<2\delta^- -\delta$ and apply the preceding estimate to a geometric  progression of radii with ratio $1+n^{-\kappa}$ covering the range $R\gg n^{-1/2+\delta}$. This progression contains $O(n^\kappa\log n)$ radii and hence the union of the exceptional sets has measure 
$O(n^\kappa \log n) n^{-2\delta^- + o(1)} = o(n^{-\delta})$.
An arbitrary admissible radius is handled by sandwiching between  two consecutive radii.

\subsection{Small-scale equidistribution: proof of \cref{thm:dim 2}}
For each $\nu>0$, choose the orthonormal basis $\cB_\nu$ in \eqref{eq:smoothed-spectral-expansion} to be a Hecke eigenbasis of $\cH_\nu(S^2)$. Applying \cref{prop:spectral second moment} then gives 
\begin{align*}
\sum_{x\in \Omega_n}(f_{R\pm\rho,\alpha}* f_{\rho,\alpha})(x) = |\Omega_n| + O\left( n^{1/2+o(1)} \sum_{\nu>0} | \widehat f_{R\pm\rho}(\nu) \widehat f_\rho(\nu) | (2\nu+1) \right).
\end{align*}
By Cauchy--Schwarz and Parseval we get 
\begin{align*}
\sum_{\nu>0} |\widehat f_{R\pm\rho}(\nu)\widehat f_\rho(\nu)|(2\nu+1) & \le
\left(\sum_{\nu\ge0}|\widehat f_{R\pm\rho}(\nu)|^2(2\nu+1)\right)^{1/2}
\left(\sum_{\nu\ge0}|\widehat f_\rho(\nu)|^2(2\nu+1)\right)^{1/2} \\
&= \mu(C_{R\pm\rho})^{-1/2}\mu(C_\rho)^{-1/2} \ll R^{-1}\rho^{-1}.
\end{align*}
Thus 
\begin{align*}
\sum_{x\in \Omega_n}(f_{R\pm\rho,\alpha}* f_{\rho,\alpha})(x) = |\Omega_n| + O\left( n^{1/2+o(1)} R^{-1} \rho^{-1} \right).
\end{align*}
Using $|\Omega_n| = n^{1 + o(1)}$ gives  
\begin{align*}
\frac{\mu(C_{R\pm\rho})}{\mu(C_R)|\Omega_n|} \sum_{x\in \Omega_n}(f_{R\pm\rho,\alpha}*f_{\rho,\alpha})(x) = 1 + O\left( R^{-1} \rho+ R^{-1} \rho^{-1} n^{-1/2+o(1)} \right).
\end{align*}
Optimizing the error term by choosing $\rho = n^{-1/4}$ we conclude with the sandwich inequality \cref{HR sandwich} that small-scale equidistribution holds whenever $R\gg n^{-1/4+o(1)}$.

\subsection{Wasserstein equidistribution: proof of \cref{thm:wasserstein}}
For $P\in \cB_\nu$, we have $\int_{S^2}P\,d\mu_n = \frac{1}{|\Omega_n|}W_n(P)$. Hence by \cref{prop:spectral second moment} and $|\Omega_n|=n^{1+o(1)}$ we have 
\begin{align*}
    \sum_{P\in \cB_\nu} \left| \int_{S^2} P\, d\mu_n \right|^2 \ll \frac{1}{|\Omega_n|^2} \sum_{P\in \cB_\nu} |W_n(P)|^2 \ll (2\nu+1) n^{-1+o(1)}.
\end{align*}
We apply \cref{prop:wasserstein-smoothing} with $\eta=\mu_n$ (for $n$ odd). Since $-\Delta$ has eigenvalue $\nu(\nu+1)$ on $H_\nu$ we obtain, for $0<t\leq 1$, 
\begin{align*}
    W_1(\mu_n,\mu) \ll t^{1/2} + \left(n^{-1+o(1)}\sum_{\nu>0}  \frac{e^{-2t\nu(\nu+1)}(2\nu+1)}{\nu(\nu+1)}\right)^{1/2}.
\end{align*}
The remaining sum satisfies
\begin{align*}
    \sum_{\nu>0} \frac{e^{-2t\nu(\nu+1)}(2\nu+1)}{\nu(\nu+1)} \ll \sum_{\nu\geq1}\frac{e^{-c t\nu^2}}{\nu} \ll \log(2/t).
\end{align*}
Choosing $t=n^{-2}$ yields the stated estimate.

We now explain why the exponent $1/2$ is best possible. Let $\eta$ be supported on a set $A$ of $N$ points on $S^2$. For $r\asymp N^{-1/2}$, the union $U_r \coloneqq  \bigcup_{x\in A} C_r(x)$ has $\mu(U_r) \leq 1/2$. For any coupling of $\eta$ and $\mu$,
\begin{align*}
    \int d_{S^2}(x,y)\, d\pi(x,y) \geq r\, \mu(S^2 \setminus U_r) \gg N^{-1/2}.
\end{align*}
Taking the infimum  over couplings yields $W_1(\eta,\mu)\gg N^{-1/2}$. Since $|\Omega_n|=n^{1+o(1)}$ (for $n$ odd), the result is optimal up to the $n^{o(1)}$ factor.

\subsection{Higher dimensional equidistribution: proof of \cref{thm:small scale}}
We restrict throughout to $n\in\mathcal N_d$, where $\mathcal N_d\subset \mathbb N$ has positive density and $|\Omega_n| \asymp n^{d-1}.$ Combining the smoothed spectral expansion \eqref{eq:smoothed-spectral-expansion} with the Weyl-sum estimate \eqref{eq:WS_bound} gives 
\begin{align*}
  \sum_{x\in \Omega_n} (f_{R\pm \rho ,\alpha }\ast f_{\rho,\alpha }) (x) = |\Omega_n| + O\left( \sum_{\nu>0} \hat f_{R\pm \rho ,\alpha }(\nu) \hat f_{\rho,\alpha }(\nu) \nu^{d-1/2+\eps}\right) n^{(d-1)/2 + \eps}.
\end{align*}
We split the sum into the ranges
$$
\nu\le R^{-1},\qquad R^{-1}<\nu\le \rho^{-1},\qquad \nu>\rho^{-1}.
$$
Applying the estimates on Gegenbauer coefficients given by \cref{prop:Bessel coeffs} then gives 
\begin{align*}
\sum_{x\in\Omega_n}(f_{R\pm\rho,\alpha}*f_{\rho,\alpha})(x) = |\Omega_n| + O\left( n^{(d-1)/2+\varepsilon} \left( R^{-d-1/2-\varepsilon} + R^{-(d+1)/2}\rho^{-d/2-\varepsilon} \right) \right).
\end{align*}
Using $\mu(C_{R\pm\rho}) = \mu(C_R) \left( 1 + O(\rho/R) \right)$ and   $|\Omega_n| \asymp n^{d-1}$ for $n \in \mathcal N_d$ implies that 
\begin{align*}
&\frac{\mu(C_{R\pm\rho})}{\mu(C_R)|\Omega_n|} \sum_{x\in \Omega_n} (f_{R\pm\rho,\alpha}\ast f_{\rho,\alpha})(x) \\
&=  1+ O\left( \frac{\rho}{R} + \left( R^{-d-1/2-\varepsilon} + R^{-(d+1)/2}\rho^{-d/2-\varepsilon} \right) n^{-(d-1)/2+\varepsilon} \right).
\end{align*}
Choosing $\rho = (R n)^{-(d-1)/(d+2)}$ and assuming $R\gg n^{-(d-1)/(2d+1)+\eta}$, then $\rho < R/2$ for sufficiently large $n$, and the error term above is $o(1)$, after choosing $\eps >0$ sufficiently small in terms of $\eta$. The sandwich inequality \cref{HR sandwich} now yields, uniformly in $\alpha$,
$$
\frac{|\Omega_n\cap C_R(\alpha)|}{|\Omega_n|\mu(C_R)} = 1 + o(1)
$$
for $n\in \mathcal N_d$ and for every $R\gg n^{-(d-1)/(2d+1)+\eta}$. Since $\eta > 0$ is arbitrary, this gives the claimed range.

\subsection{Conditional admissible levels: proof of \cref{thm:BRS all levels}}
For an admissible $N$ write $N=4^aM$ with $4\nmid M$. Repeated reduction modulo $4$ shows that any solution of $x^2+y^2+z^2 = N$ is of the form $(x,y,z) = 2^a (x',y',z')$ with $x'^2+y'^2+z'^2 = M$. Since $(x,y,z)/ \sqrt N = (x',y',z')/ \sqrt M$, we have $\mathcal E_N = \mathcal E_M$. Since $N$ is admissible, Legendre's three-square theorem implies $M \not\equiv 7 \pmod 8$. Write $M=dq^2$ with $d$ squarefree. Since $4\nmid M$, the integer $q$ is odd. Hence $q^2 \equiv 1 \pmod 8$ and so $M = dq^2 \equiv d \pmod 8$. Hence $d \not\equiv 7 \pmod 8$.
Let \(P\in H_\nu(S^2)\) be an \(L^2\)-normalized Hecke eigenfunction with $\nu > 0$ even. (For odd $\nu$, the Weyl sum vanishes by antipodal symmetry.) Using the bound \eqref{eq:GLH:square-bound} we obtain  
    \begin{align*}
|r_P(M)| = |r_P(dq^2)| \ll d^{\nu/2+1/4+\eps}\nu^\eps q^{\nu+1/2+o(1)} \ll M^{\nu/2+1/4+\eps} \nu^\eps.
\end{align*}
Therefore
    $$
    \left| \sum_{\xi \in  \cE_N} P(\xi) \right| = M^{-\nu/2} |r_P(M)| \ll  M^{1/4+\eps} \nu^\eps \ll N^{1/4+\eps} \nu^\eps.
    $$
This is the Weyl-sum input used in the proof of the variance estimate of \cite[Thm 1.7]{bourgain2016spatial}, so the same argument extends to every admissible $N$.

\section{Covering and Diophantine applications}

\subsection{Rational points of bounded height}

We can also examine the small-scale equidistribution for the point sets $\Omega_T$ consisting of all rational points on $S^d$ of height up to $T$.

\begin{prop}
Let $\Omega_T$ denote the rational points on $S^d$ of height at most $T$. For $d\not\equiv 1 \pmod 4$,
$$
\frac{|\Omega_T\cap C_R(\alpha)|}{|\Omega_T|\mu(C_R)} = 1 + O\left( R^{-2d/(d+1)}T^{-d/(d+1)}\log^{2\eta/(d+1)}T\right)
$$
is valid whenever $R\gg T^{-1/2}\log^{\eta/d}T$.
\end{prop}

\begin{proof}
 In  \cite[Thm 8.4]{BurrinGroebner2024} we established using contour integration of $L$-functions the upper bound
\begin{align*} 
    \sum_{x\in \Omega_T} P(x) \ll \|P\|_\infty \log^\eta (T) T^{d/2}
\end{align*}
for $d\not\equiv 1$ (mod 4), with $\eta=1/2$ if $d>2$ and $\eta=3/2$ if $d=2$.\footnote{ For $d\equiv 1$ (mod 4), the same method gives only the weaker bound $$ \sum_{x\in \Omega_T} P(x) \ll_\eps \|P\|_\infty T^{(d+1)/2+\eps}.$$  }   Reworking the proof of \cref{thm:small scale} with this estimate, we obtain the effective equidistribution result
\begin{align*}
    \frac{|\Omega_T\cap C_R(\alpha)|}{|\Omega_T|\mu(C_R)} = 1 + O\op{R^{-2d/(d+1)}T^{-d/(d+1)}\log^{2\eta/(d+1)}(T)},
\end{align*}
valid whenever $R\gg T^{-1/2}\log^{\eta/d}(T)$. 
\end{proof}
This result improves the speed of equidistribution in spherical caps obtained by \cite{KelmerYu2023} by a factor of $2$. Their approach, based on counting with Siegel transforms, has the remarkable feature that when $d\not\equiv 1$ (mod 8) their equidistribution results hold for the full shrinking range $R \gg T^{-1/2}$. Similarly, the variance computation in the proof of  \cref{thm:dim 2} establishes that small-scale equidistribution for $\Omega_T$ holds for each dimension $d\not\equiv1$ (mod 4) and almost every spherical cap whenever $R\gg T^{-1+o(1)}$.

These results complement the covering radius discussion of $\Omega_T$; we now prove \cref{thm:Diophantine}.

\begin{proof}[Proof of \cref{thm:Diophantine}]
The repulsion that is  observed numerically about rational points of small denominator in \cref{fig:Omega_T} can be quantified arithmetically as follows.  If $x=\tfrac{m}{n}$ and $x'=\tfrac{m'}{n'}$ are distinct rational points on $S^d$, written in lowest terms, then  we have the elementary computation
\begin{align}\label{eq: arithmetic repulsion}
|x-x'|^2 = 
\frac{2(nn'-\scal{m,m'})}{nn'} \geq \frac{2}{nn'}.
\end{align} 
 In particular, if $x'$ is a fixed rational point of bounded height (e.g. the north pole of the sphere), then every $x \in \Omega_T \setminus \{x'\}$ satisfies $|x-x'| \gg T^{-1/2}$. Let $y$ be the midpoint of a shortest geodesic joining $x'$ to the nearest point of $\Omega_T \setminus \{ x'\}$. Then $\mathrm{dist}(y, \Omega_T) \gg T^{-1/2}$, which implies that $K((\Omega_T)) \ge 2$.  
In the other direction, Dirichlet's theorem for rational points on the sphere \cite[Thm 4.1]{KleinbockMerrill2015} asserts the existence of a positive constant $C$ that for any $\alpha \in S^d$, $T>0$, we have $\Omega_T\cap C_R(\alpha )\neq\emptyset$ for $R\geq C T^{-1/2}$; hence $K((\Omega_T))\leq 2$.

We now turn to the generic covering exponent. 
  Let $(R_T)_T$ be a sequence of radii such that
$$
\limsup_{T \to \infty} \frac{\log |\Omega_T|}{-\log \mu(C_{R_T})} < 1.
$$
Then there exists $0<\eta<1$ such that for all sufficiently large $T$ we have $\mu(C_{R_T}) \leq |\Omega_T|^{-1/(1-\eta)}$. Hence
$$
\mu \left( \bigcup_{x\in\Omega_T} C_{R_T}(x) \right) \leq |\Omega_T| \mu(C_{R_T}) \leq |\Omega_T|^{-\eta/(1-\eta)} \to 0
$$
as $T \to \infty$. Thus no such sequence $(R_T)_T$ can cover asymptotically almost all of $S^d$ and therefore $K_\mu((\Omega_T)) \geq 1$.  
The corresponding upper bound $K_\mu\leq 1$ can be derived from the counting result \cite[Thm 1.8(b)]{KelmerYu2023}.
\end{proof}

\subsection{Sums of two squares and a mini-square: proof of \cref{thm:Linnik}}

The pointwise smoothing argument of Section 4.4 applies equally if we replace $C_R(\alpha )$ by the annulus $A_{r,R}(\alpha)=\{ x\in S^2 : r<|x-\alpha|<R\}.$
Let $\alpha$ be the north pole. The equatorial strip $A=\{x=(x_1,x_2,x_3)\in S^2: |x_3|<n^{-\delta}\}$ is such an annulus, with $r^2=2(1-n^{-\delta})$ and $R^2=2(1+n^{-\delta}).$ Then $\mu(A)\asymp n^{-\delta}$ and
$$
|\Omega_n\cap A|=|\Omega_n|\mu(A)(1+O(n^{-(1-3\delta)/4+o(1)}))=|\Omega_n|\mu(A)(1+o_\delta(1)) 
$$
as long as $\delta<1/3$. Hence for any odd square $\ell=n^2$ there exist $x,y,z\in\Z$ coprime such that $x^2+y^2+z^2=\ell$ with $|z| \ll n^{1-\delta} \ll_\eps l^{1/3+\eps}$.

\subsection{Conditional counterpart: proof of \cref{thm:conditional Linnik}}
Let $\cE_n^{\rm prim}=\{x\in \Z^3_{\rm prim} : |x|=n\}$. In particular $|\cE_n^{\rm prim}|=r^*_3(n)\gg_\eps n^{1/2-\eps}$. 

Using \cref{prop:bounds on coeffs dim 2} and its corollary, we have 
\begin{align*}
    |r_P^*(d\ell^2)| \ll |r_P(d)|\, \ell^{\nu+1/2+o(1)}.
\end{align*}
On the other hand, Waldspurger's formula, in the form \cite[(5.4)]{bourgain2016spatial} gives 
\begin{align*}
    |r_P(d)|^2 =     c \frac{d^{\nu+1/2}L(1/2,f)L(1/2,f\times \chi_{D_d})}{L(1,{\rm Sym}^2 f)},
\end{align*}
where $f$ is the holomorphic Hecke cusp form of weight $2+2\nu$ associated to $P$ and 
\begin{align*}
    D_d =  \begin{cases}
        -4d, & d\equiv 1,2 \pmod 4,\\
        -d,  & d\equiv 3 \pmod 4.
        \end{cases}
\end{align*}
By GLH, together with the Hoffstein--Lockhart lower bound for $L(1,{\rm Sym}^2 f)$, this implies 
\begin{align*}
    |r_P(d)| \ll d^{\nu/2+1/4+\eps} \nu^\eps.
\end{align*}
Consequently, for $n=d\ell^2$,
\begin{align*}
    \left| \sum_{x\in \cE^{\rm prim}_{n}} P(x) = n^{-\nu/2}|r_P^*(n)| \right| \ll n^{-\nu/2} d^{\nu/2+1/4+\eps} \ell^{\nu+1/2+o(1)} \ll n^{1/4+\eps}.
\end{align*}
Let $A\subset S^2$ be measurable and put $f_A= {\bf 1}_A/\mu(A)$. Using the pre-trace formula 
$
\sum_{P\in \mathcal{B}_\nu} P(x)^2 = 2\nu+1
$
as in the proof of \cref{thm:dim 2}, and then applying Cauchy--Schwarz and Parseval, the preceding bound gives
$$
\sum_{x\in \cE^{\rm prim}_{n}} f_{A}\ast f_{\rho,\alpha}(x) = |\cE^{\rm prim}_{n}| + O\left( n^{1/4+\eps}\big(\mu(A) \mu(C_\rho)\big)^{-1/2}\right),
$$
where $f_{\rho,\alpha}$ is the normalized indicator of the spherical cap $C_\rho(\alpha)$.  Fix $\eta>0$, set $h=n^{-1/6+\eta}$, $\rho=\frac h{10}$ and let $A_h=\{u\in S^2:|u_3|\le h\}.$ Since $\mu(A_h)\asymp h$ and $\mu(C_\rho)\asymp h^2$, the smoothed estimate gives
$$
\sum_{x\in\cE^{\rm prim}_{n}}(f_{A_h}*f_\rho)(x) = |\cE^{\rm prim}_{n}|+O_\varepsilon(n^{1/4+\varepsilon}h^{-3/2}) = |\cE^{\rm prim}_{n}|+o(|\cE^{\rm prim}_{n}|),
$$
after choosing $\varepsilon>0$ sufficiently small. Hence there exists $x\in\cE^{\rm prim}_{n}$ with $(f_{A_h}*f_\rho)(x)>0$, so $x$ lies within distance $O(\rho)$ of $A_h$. Writing $x=(x_1,x_2,z)/\sqrt n$, we get $\frac{|z|}{\sqrt n}\ll h+\rho\ll h,$ and therefore $|z|\ll \sqrt n h=n^{1/3+\eta}.$ Since $\eta>0$ is arbitrary, this proves $|z|\ll_\varepsilon n^{1/3+\varepsilon}$.

\appendix

\section{Size of rational points of fixed height} \label{app:size of Omega_n}
Let $Q(x)=\frac12 x^T A x$ be a positive definite quadratic form in $d+1$ variables. Assume that $A$ is symmetric with integral entries and even diagonal entries. Assume moreover that there exists a positive integer $N$, equal to $4$ times a squarefree odd integer, such that $NA^{-1}$ has integral entries and even diagonal entries. Let
$$
\mathcal E_Q := \{x\in \mathbb R^{d+1}:Q(x)=1\}
$$
be the associated ellipsoid. For $n\ge 1$, let $\Omega_n$ denote the set of rational points on $\mathcal E_Q$ of height exactly $n$.

\begin{thm}\label{thm:size of Omega_n}
Assume $d\ge 4$. There exists an arithmetic progression $\mathcal S\subset \mathbb N$ such that
$$
|\Omega_n| \asymp_Q n^{d-1}
$$
for all $n\in \mathcal S$.
\end{thm}

Let $L=(\mathbb Z^{d+1},Q)$ be the lattice attached to $Q$. Put $m=d+1$ and $k=\frac m2.$ Thus $m\ge 5$ and $m-2=d-1$. The rational points of height $n$ are counted by primitive integral solutions of $Q(x)=n^2.$ Accordingly, define
$$
r_Q^{\mathrm{prim}}(n^2) := \# \{x \in L: Q(x)=n^2, x \text{ primitive} \}.
$$
It is enough to prove that $r_Q^{\mathrm{prim}}(n^2)\asymp_Q n^{m-2}$ for all $n$ in some arithmetic progression.

\textbf{Step 1: The genus main term.}
Recall that
$$
\Theta_Q(z) = \sum_{x \in L} q^{Q(x)} = \sum_{n=0}^\infty r_Q(n) q^n  \in M_k(N,\chi).
$$
Let $\mathrm{gen}(L)$ be the genus of $L$ and define its mass
$$
\mathrm{mass}(\mathrm{gen}(L)) = \sum_{[L'] \in \mathrm{gen}(L)} \frac{1}{|\mathrm{Aut}(L')|}.
$$
The weighted average of the theta series over the genus 
$$
\Theta_{\mathrm{gen}(L)}(z) = \frac{1}{\mathrm{mass}(\mathrm{gen}(L)) } \sum_{[L'] \in \mathrm{gen}(L)} \frac{\Theta_{L'}(z)}{|\mathrm{Aut}(L')|} =: \sum_{n=0}^\infty r_{\mathrm{gen}(L)}(n) q^n.
$$
is a modular form of weight $k$ and level $N$. Siegel's formula \cite{SchulzePillot2004Survey} says that 
$$
r_{\mathrm{gen}(L)}(n) = c_\infty(L) n^{k-1} \prod_{p} \alpha_p(L,n),
$$
where $c_\infty(L) > 0$ and the $p$-adic local density  $\alpha_p(L,n)$  is defined by 
$$
\alpha_p(L,n) = \lim_{a\to\infty} p^{-a(m-1)} \, \#\{x\in (L/p^a L) : Q(x) \equiv n \pmod{p^a} \};
$$
see \cite{Yang1998}. Moreover, $\Theta_Q- \Theta_{\mathrm{gen}(L)} \in S_k(N,\chi)$ is cuspidal. The decomposition $\Theta_Q = \Theta_{\mathrm{gen}(L)} +(\Theta_Q - \Theta_{\mathrm{gen}(L)})$ separates the Eisenstein and cuspidal parts. Siegel formula gives the main term, while the cusp contribution at square indices is $O_Q( n^{k-1+\varepsilon})$ by Deligne's bound in integral weight, and by Shimura's correspondence followed by Deligne's bound in half-integral weight. Hence 
$$
r_Q(n^2) = c_\infty(L) n^{2k-2} \prod_{p} \alpha_p(L,n^2) + O_Q( n^{k-1+\varepsilon}).
$$

\textbf{Step 2: Passing to primitive representations.}
By Möbius inversion we have 
$$
r_Q^{\mathrm{prim}}(n^2) = \sum_{\delta\mid n} \mu(\delta)\,r_Q(n^2/\delta^2).
$$
Applying the previous estimate to $(n/\delta)^2$ gives
$$
r_Q^{\mathrm{prim}}(n^2) = c_\infty(L)n^{m-2} \sum_{\delta\mid n} \mu(\delta)\delta^{-(m-2)} \prod_p\alpha_p(L,n^2/\delta^2) + O_Q(n^{k-1+\varepsilon}).
$$

\textbf{Step 3: Factorization of the primitive genus term.}
We first record unit-square invariance of the local densities. If $u\in \mathbb Z_p^\times$, then multiplication by $u$ gives a bijection
$$
L/p^eL \to L/p^eL,\qquad x \mapsto ux.
$$
Since $Q(ux)=u^2Q(x),$ this bijection identifies the solutions of
$$
Q(x)\equiv t\pmod {p^e}
$$
with the solutions of
$$
Q(x)\equiv u^2t\pmod {p^e}.
$$
Hence $\alpha_p(L,u^2t)=\alpha_p(L,t)$ for every $u\in\mathbb Z_p^\times$.

Now let $\delta\mid n$ be squarefree and write $\delta=\prod_{p\mid n} p^{\epsilon_p}$ with $\epsilon_p \in \{0,1\}.$ For a fixed prime $p$, the factors coming from primes $q\ne p$ are $p$-adic unit squares. Hence, by unit-square invariance we have $\alpha_p(L,n^2/\delta^2) = \alpha_p(L,n^2/p^{2\epsilon_p}).$ We also have 
$$
\mu(\delta) \delta^{-(m-2)}  = \prod_{p \mid n} (-p^{-(m-2)})^{\epsilon_p}.
$$
It follows then that 
$$
\sum_{\delta\mid n} \mu(\delta)\delta^{-(m-2)} \prod_p\alpha_p(L,n^2/\delta^2) = \prod_p\beta_p(L,n^2),
$$
where
$$
\beta_p(L,n^2) :=
\begin{cases}
\alpha_p(L,n^2), & p\nmid n, \\ 
\alpha_p(L,n^2)-p^{2-m}\alpha_p(L,n^2/p^2), & p\mid n.
\end{cases}
$$
Consequently we have 
$$
r_Q^{\mathrm{prim}}(n^2) = c_\infty(L)n^{m-2} \prod_p\beta_p(L,n^2) + O_Q(n^{k-1+\varepsilon}).
$$

\textbf{Step 4: Identification with primitive local densities.}
Define the primitive local density  
$$
\alpha_p^{\mathrm{prim}}(L,n^2) := \lim_{e\to\infty} p^{-e(m-1)}
\#\big\{x\in L/p^eL: Q(x) \equiv n^2 \pmod{p^e}, \,  p\nmid x \text{ if } p\mid n\big\}. 
$$
If $p\nmid n$, the condition $p\nmid x$ is automatic, so $\alpha_p^{\mathrm{prim}}(L,n^2)=\alpha_p(L,n^2).$
Now suppose $p\mid n$. Let 
$$
R_{p^e}(n^2) = \{x\in L/p^eL: Q(x)\equiv n^2\pmod {p^e} \},
$$
and let $R_{p^e}^{\mathrm{div}}(n^2)$ be the subset of such classes with $p\mid x$.  Multiplication by $p$ gives a bijection
$$
L/p^{e-1}L \xrightarrow{\sim} pL/p^eL, \qquad y \mapsto py.
$$
Under this bijection, for $e\ge 2$, the set $R_{p^e}^{\mathrm{div}}(n^2)$ corresponds to the set of classes $y\in L/p^{e-1}L$ satisfying
$$
Q(y)\equiv (n/p)^2\pmod {p^{e-2}}. 
$$ 
Since each class modulo $p^{e-2}$ has $p^m$ lifts modulo $p^{e-1}$ it follows that 
$$
\# R_{p^e}^{\mathrm{div}}(n^2) = p^m \#R_{p^{e-2}}((n/p)^2).
$$
After multiplying by $p^{-e(m-1)}$ and passing to the limit, we obtain
$$
\alpha_p^{\mathrm{prim}}(L,n^2) = \alpha_p(L,n^2) - p^{2-m}\alpha_p(L,(n/p)^2).
$$
Therefore $\alpha_p^{\mathrm{prim}}(L,n^2)=\beta_p(L,n^2)$ for all primes $p$ and all $n$.

\textbf{Step 5: A uniform lower bound for the local product.}
We now prove, after restricting $n$ to a suitable arithmetic progression, that
$$
\prod_p\beta_p(L,n^2) = \prod_p\alpha_p^{\mathrm{prim}}(L,n^2)
$$
is bounded below by a positive constant depending only on $Q$. 

Let $p\nmid N$ be a good prime. Thus $p$ is odd and the reduction $\overline Q$ is nondegenerate over $\mathbb F_p$.  Put $a=\overline{n^2} \in \mathbb F_p.$ For $e\ge 1$, set
$$
S_e := \{x \in L/p^eL: Q(x) \equiv n^2 \pmod {p^e}, p \nmid x \text{ if } p\mid n \}.
$$
Also define
$$
Z_p(a) := \left\{ x \in \mathbb F_p^m: \overline Q(x)=a, x \ne 0 \text{ if } a=0 \right\}.
$$
Since $a=0$ exactly when $p\mid n$, reduction modulo $p$ gives a map $ S_e \to Z_p(a)$. Every element of $Z_p(a)$ is nonsingular. Hence Hensel lifting implies that each element of $Z_p(a)$ has exactly $p^{(e-1)(m-1)}$ lifts to an element of $S_e$. Hence $\# S_e =p^{(e-1)(m-1)}\#Z_p(a).$ Therefore 
$$
\alpha_p^{\mathrm{prim}}(L,n^2) = \lim_{e\to\infty} p^{-e(m-1)} \#S_e = p^{-(m-1)}\#Z_p(a).
$$

By the standard finite-field estimate for a nondegenerate quadratic form in $m$ variables \cite[Theorems 6.26 and 6.27]{LidlNiederreiter1997}, we have $\#Z_p(a) = p^{m-1}+O_L(p^{m/2}).$ Since $m\ge 5$, we have
$$
\alpha_p^{\mathrm{prim}}(L,n^2)=1+O_L(p^{-3/2})
$$
uniformly in $n$ for all good primes $p\nmid N$. Consequently, there exists a constant $A>0$ such that
$$
\left|\alpha_p^{\mathrm{prim}}(L,n^2)-1\right| \le \frac{A}{p^{3/2}}
$$
for all $p\nmid N$ and all $n$. Choose $P_0$ large enough such that $\frac{A}{p^{3/2}} \le \frac12$ for all primes $p>P_0$. Let
$$
\mathcal T := \{p : p \mid N\} \cup \{ p:p \le P_0 \}.
$$
Then we have 
$$
\prod_{p\notin\mathcal T} \alpha_p^{\mathrm{prim}}(L,n^2) \ge \prod_{p\notin\mathcal T} \left(1-\frac{A}{p^{3/2}}\right) =: C_\infty(L,\mathcal T)>0.
$$
The positivity follows because $ \sum_{p \notin \mathcal T} \frac{A}{p^{3/2}} < \infty.$ 

It remains to control the finitely many primes in $\mathcal T$. The height zeta function $R_Q(s)=\sum_{n\ge 1} |\Omega_n| n^{-s}$ has a simple pole at $s=d$, see \cite{BurrinGroebner2024}. This guarantees the existence of a rational point on $\mathcal E_Q$. Writing it in primitive integral form gives a vector $x_0\in L$ and an integer $n_0\ge 1$ such that $Q(x_0)=n_0^2$ and $x_0$ is primitive. Hence $\Omega_{n_0} \ne \emptyset$. Since $x_0$ is primitive over $\mathbb Z$, it is primitive over $\mathbb Z_p$ for every prime $p$. Hence $n_0^2$ is primitively represented over $\mathbb Z_p$ for every $p$, and therefore $\alpha_p^{\mathrm{prim}}(L,n_0^2)>0$ for every prime $p$.
Now set
$$
M = \prod_{p\in\mathcal T}p^{v_p(n_0)+1}.
$$
If $n\equiv n_0\pmod M,$ then for every $p\in\mathcal T$ we have $v_p(n)=v_p(n_0).$ Indeed, writing $n_0 = p^{v_p(n_0)} u_0$ with $u_0 \in \mathbb Z_p^\times$, the congruence $n\equiv n_0\pmod {p^{v_p(n_0)+1}}$ implies $ n/p^{v_p(n_0)} \equiv u_0\pmod p,$ so $p\nmid n/p^{v_p(n_0)}$.
Thus, for every $p\in\mathcal T$ we have $u_p := n/n_0 \in \mathbb Z_p^\times.$ By unit-square invariance we have 
$$
\alpha_p^{\mathrm{prim}}(L,n^2) = \alpha_p^{\mathrm{prim}}(L,u_p^2n_0^2) = \alpha_p^{\mathrm{prim}}(L,n_0^2),
$$
and therefore
$$
\prod_{p\in\mathcal T} \alpha_p^{\mathrm{prim}}(L,n^2) = \prod_{p\in\mathcal T} \alpha_p^{\mathrm{prim}}(L,n_0^2) =: C_{\mathcal T}(Q,n_0)>0.
$$
Combining the finite and infinite parts, we obtain
$$
\prod_p\beta_p(L,n^2) = \prod_p\alpha_p^{\mathrm{prim}}(L,n^2) \ge C_{\mathcal T}(Q,n_0) C_\infty(L,\mathcal T) =:C(Q) > 0
$$
for every $n\equiv n_0\pmod M$.

For the upper bound, the same estimates give
$$
\prod_{p\notin\mathcal T} \alpha_p^{\mathrm{prim}}(L,n^2) \le \prod_{p\notin\mathcal T} \left( 1+\frac{A}{p^{3/2}} \right) < \infty.
$$
Moreover, for $p\in\mathcal T$ and $n\equiv n_0\pmod M$ we have 
$$
\prod_{p\in\mathcal T}\alpha_p^{\mathrm{prim}}(L,n^2) = \prod_{p\in\mathcal T}\alpha_p^{\mathrm{prim}}(L,n_0^2) < \infty.
$$
Together with the lower bound already proved, we have
$$
1 \ll_Q \prod_p \alpha_p^{\mathrm{prim}}(L,n^2) \ll_Q 1.
$$

\textbf{Step 6: Concluding.}
Recall that 
$$
r_Q^{\mathrm{prim}}(n^2) = c_\infty(L) \left(\prod_p \alpha_p^{\mathrm{prim}} (L,n^2)\right)n^{m-2} + O_Q(n^{k-1+\varepsilon}).
$$
Choose $\varepsilon>0$ such that $k-1+\varepsilon<m-2.$ Since the local product is bounded above and below on the arithmetic progression, the main term is $\asymp_Q n^{m-2}$, while the error term is $o(n^{m-2})$. Hence, after replacing this congruence class by a tail arithmetic progression 
$$ \mathcal S=\{n_1+\ell M:\ell\ge 0\}, \qquad n_1\equiv n_0\pmod M, 
$$
we obtain
$$
r_Q^{\mathrm{prim}}(n^2)\asymp_Q n^{m-2} = n^{d-1},
$$
and hence $|\Omega_n|\asymp_Q n^{d-1}$ for all $n \in \mathcal S$.

\section{Proof of the Eichler commutation relation} \label{appendix:Eichler}

The goal of this section is to give a self-contained proof of the Eichler commutation relation \eqref{Eichler_commutation_relation}. We follow closely the arguments in \cite{Ponomarev81,bocherer1994mellin}. 

 Fix an odd prime $p$ throughout this appendix.  A $p$-neighbor of a lattice $L \subset B$ is a lattice $K$ such that $pK \subset L$ and $K\neq L$. The localization of a $\Z$-lattice $L$ at $p$ is the $\Z_p$-lattice $L_p = L \otimes_\Z \Z_p$. The genus of a lattice $L$ consists of all lattices $K$ in $B$ which are (everywhere) locally isometric to $L$, that is, if and only if for all primes $\ell$ we can write $L_\ell = \alpha_\ell K_\ell \alpha_\ell^{-1}$ for some $\alpha_\ell \in B_\ell^\times$. We write $K \sim L$ if $K$ is a $p$-neighbor of $L$ and belongs to the same genus as $L$.

Let $\Lambda$ be as in \cref{subsection:thetaseries_quaternion}. Central to the argument is to count the $p$-neighbors of $\Lambda$ that lie in the same genus as $\Lambda$. This problem can be studied locally.  Since $p \neq 2$, we have $(\Z+2\mathcal O)_p=\mathcal O_p$. Moreover, $B_p \cong M_2(\Q_p)$ and we choose such an isomorphism so that the
maximal order $\mathcal O_p$ is identified with $M_2(\Z_p)$. If $K \sim \Lambda$, then the corresponding local order is of the form $\alpha_p^{-1} M_2(\Z_p) \alpha_p$ for some $\alpha_p \in B_p^\times \simeq M_2(\Q_p)^\times$. The condition that $pK\subset\Lambda$ gives
$$
p \alpha_p^{-1} M_2(\Z_p)\alpha_p \subset M_2(\Z_p) ,
$$
and the condition $K \neq \Lambda$ gives 
$$
\alpha_p^{-1} M_2(\Z_p) \alpha_p \neq M_2(\Z_p).
$$
 The first condition implies that, after scaling by an element of $\Q_p^\times$, we may write $\alpha_p \in M_2(\Z_p)$ with $\det(\alpha_p) = u$ or $pu$ for some unit $u \in \Z_p^\times$. The second condition excludes the first case. Moreover, left multiplication by an element of $GL_2(\Z_p)$ does not change the corresponding local order $\alpha_p^{-1} M_2(\Z_p) \alpha_p$. Hence, multiplying on the left by $\mathrm{diag}(u^{-1},1) \in GL_2(\Z_p)$ if necessary, we may assume that $\det(\alpha_p)=p$. Every such class admits a representative of the form 
\begin{align}  \label{M2(Zp)_detp}
\left\lbrace
\begin{pmatrix}
    p^a & c \\
    0 & p^b 
\end{pmatrix} \mid a,b \in \Z_{\geq 0} \text{ with } a+b=1, c \in \Z/(p^b) \right\rbrace;
\end{align}
see \cite[Chapter II, Theorem 2.3]{Vigneras1980}. This gives a set of $p+1$ representatives for these classes.   

In the following we write $R(n,\Lambda) = \{ v \in \Lambda : \mathrm{nr}(v) = n \}$ and recall that 
$$
r_{\Lambda, P}(n) = \sum_{v \in R(n,\Lambda)} P(v).
$$

\begin{lm} \label{lm:nr_cases}
Let $p\neq 2$ be prime and fix $v \in R(p^2n,\Lambda)$. Then
    \begin{align*}
     \pi(v,\Lambda) &:= \# \{ p\text{-neighbors } K \text{ of } \Lambda \mid v \in pK  ,\,  K       \in \text{gen}(\Lambda) \} \\ 
         &= 
         \begin{cases*}
            1 & if  $v \in \Lambda$ and $v \not\in p\Lambda$; \\
            1 + \left(\dfrac{-n}{p}\right) & if $v \in p\Lambda$ and $v \not\in     p^2\Lambda$; \\
            p+1 & if  $v \in p^2\Lambda$.
        \end{cases*}
    \end{align*}
\end{lm}

\begin{proof}
We only prove the case $v \in p\Lambda \setminus p^2\Lambda$; the proof of the other two cases is similar. After passing to the localization at $p$, we may suppose that $v \in p M_2(\Z_p)$ and $v \not\in p^2 M_2(\Z_p)$.  After conjugation by an element of $GL_2(\Z_p)$,  we may assume that 
$$
v = p \begin{pmatrix}
    0 & n \\
    -1 & 0
\end{pmatrix}  .
$$
Then $ \pi(v,\Lambda)$ is the number of elements $\alpha_p$ as in \cref{M2(Zp)_detp}  that satisfy
\begin{align} \label{v_pneighbors}
\alpha_p \begin{pmatrix}
    0 & n \\
    -1 & 0
\end{pmatrix}  \alpha_p^{-1}  \in M_2(\Z_p) .
\end{align}
A direct computation shows that \cref{v_pneighbors} is satisfied if and only if $\alpha_p$ as in \cref{M2(Zp)_detp} satisfies $b=1$ and $c^2 + n \equiv 0 \,\,  (\text{mod } p)$. This quadratic equation modulo $p$ has  
$$
1 + \left(\dfrac{\mathrm{disc}(c^2+n)}{p}\right) = 1 + \left(\dfrac{-4n}{p}\right) = 1 + \left(\dfrac{-n}{p}\right) 
$$
solutions.
\end{proof}
We now prove the Eichler commutation relation. According to \cref{lm:nr_cases} we can write
\begin{align} \label{eq:pivL}
    \sum_{v \in R(p^2 n,\Lambda)} P(v) \pi(v, \Lambda) &= r_{\Lambda, P}(p^2 n) + \left(\dfrac{-n}{p}\right)  \sum_{\substack{ v \in R(p^2 n,\Lambda) \\ v \in p\Lambda  \\ v \not\in p^2\Lambda}}  P(v)  +  p \sum_{\substack{v \in R(p^2 n,\Lambda) \\ v \in p^2\Lambda}}  P(v) .
\end{align}
For $p^2 \mid n$,  the right side of \cref{eq:pivL} equals 
\begin{align*}
r_{\Lambda, P}(p^2 n) +  p^{1 + 2 \nu} \, r_{\Lambda, P}(n/p^2)
\end{align*}
(where $r_{\Lambda, P}(n/p^2)= 0 $ if $p^2 \nmid n$), while for $p \nmid n$ it equals 
\begin{align*}
r_{\Lambda, P}(p^2 n) +  p^\nu \left(\dfrac{-n}{p}\right) r_{\Lambda, P}(n) .
\end{align*}

On the other hand, expanding the definition of $\pi$ gives 
\begin{align*}
 \sum_{v \in R(p^2 n,\Lambda)} P(v) \pi(v, \Lambda)  &=  \sum_{K \sim \Lambda}  \sum_{v \in R(p^2n,\Lambda) \cap pK} P(v).
\end{align*}
Since the Hurwitz order $\mathcal O$ has class number one, the local
classification of $p$-neighbors globalizes: every $p$-neighbor $K \sim \Lambda$ is of the form $K=\gamma \Lambda \gamma^{-1}=p^{-1} \gamma \Lambda \bar{\gamma}$ for some $\gamma \in \mathcal O$ with $\mathrm{nr}(\gamma)=p$. The element $\gamma$ is unique up to right multiplication by $\mathcal O^\times$. For such $K$, the map $v \mapsto \gamma v\bar{\gamma}$ gives a bijection between $R(n,\Lambda)$ and $R(p^2n,\Lambda) \cap pK$. Because $P$ is homogeneous of degree $\nu$ we have $P(\gamma v\bar{\gamma}) = p^\nu P(\gamma v \gamma^{-1})$. Therefore 
\begin{align*}
 \sum_{v \in R(p^2 n,\Lambda)} P(v) \pi(v, \Lambda)   &= \frac{p^\nu}{| \mathcal{O}^\times| } \sum_{\substack{\gamma \in \mathcal{O} \\  \mathrm{nr}(\gamma) = p }}  \sum_{v \in R(n,\Lambda)} P(\gamma v \gamma^{-1}) \\
  &=  p^\nu \sum_{v \in R(n,\Lambda)} (\tilde{T}_p P)(v).
\end{align*}
 Recalling that the Fourier coefficients $b(n)$ of the image of a modular form $\sum a(n) q^n \in M_{3/2 + \nu}$ under the Hecke operator $T_{p^2}$ satisfy 
$$
 b(n) =  a(p^2 n)  +  p^\nu \left(\dfrac{-n}{p}\right) a(n)  +  p^{1 + 2 \nu} a(n/p^2)
$$
(where again $a(n/p^2) = 0 $ if $p^2 \nmid n$), we conclude from the above expressions that  $T_{p^2} \Theta_{\Lambda,P} = p^\nu \Theta_{\Lambda, \tilde{T}_p P}$. 
\bibliographystyle{alpha}
\bibliography{biblio}

\end{document}